\documentclass[11pt,leqno]{article}

\usepackage[all]{xy}  
\usepackage{amssymb,amsmath,amsthm,    url,rotating}
\usepackage{mathrsfs}
  \usepackage{graphicx,epsfig}
  \usepackage{xcolor,graphicx}
  
    \usepackage{makeidx}   
\newcommand{\n}{\noindent}

\newcommand{\vp}{\varepsilon}
\newcommand{\bb}[1]{\mathbb{#1}}
\newcommand{\cl}[1]{\mathcal{#1}}

\newcommand{\ovl}{\overline}

\theoremstyle{plain}
\newtheorem{thm}{Theorem}[section]
\newtheorem{lem}[thm]{Lemma}

\newtheorem{pro}[thm]{Proposition}

\newtheorem{cor}[thm]{Corollary}

\theoremstyle{definition}

\newtheorem{dfn}[thm]{Definition}

\theoremstyle{remark}
\newtheorem{rem}[thm]{Remark}

\numberwithin{equation}{section}

\def\tilde{\widetilde}

\renewcommand{\tilde}{\widetilde}

\def\C{\bb  C}
\def\CC{\bb  C}

\def\F{\bb  F}

\def\d{\delta}
\def\NN{\bb  N}
\def\N{\bb  N}

\def\CC{\bb  C}

\def\CC{\bb  C}

\def\F{\bb  F}

\def\d{\delta}
\def\NN{\bb  N}

\def\Z{\bb  Z}

\def\CC{\bb  C}

\def\phi{\varphi}

\def\n{\noindent}
\def\nl{\nolimits}

\begin{document}

\title{On the    Lifting  Property for $C^*$-algebras}

\author{by\\
Gilles  Pisier\footnote{ORCID    0000-0002-3091-2049}   \\
Texas  A\&M  University\\
College  Station,  TX  77843,  U.  S.  A.}

\def\C{\mathscr{C}}
\def\B{\mathscr{B}}
\def\I{\cl  I}
\def\e{\cl  E}
 
  \def\a{\alpha}
 
  \maketitle
\begin{abstract}    
We  characterize the lifting property (LP) of a separable
$C^*$-algebra $A$
by a property of its maximal tensor
product with other $C^*$-algebras, namely we prove that $A$
has the LP if and only if for any family $\{D_i\mid i\in I\}$ of
$C^*$-algebras
the canonical map
$$   {\ell_\infty(\{D_i\}) \otimes_{\max} A}\to {\ell_\infty(\{D_i \otimes_{\max} A\}) }$$
is isometric. 
Equivalently, this holds
if and only if
$M \otimes_{\max} A= M \otimes_{\rm nor} A$ for any von Neumann algebra $M$.
 \end{abstract}

  
  \tableofcontents
  
    \vfill \eject
    A separable $C^*$-algebra $A$ has the lifting property (LP in short)
    if any contractive completely positive (c.c.p. in short)  map $u: A \to C/\cl I$ into a quotient $C^*$-algebra
    admits a c.c.p. lifting $\hat u: A \to C$. 
    In \cite{[CE5]} Choi and Effros proved that all (separable) nuclear $C^*$-algebras have the LP.
    Later on in \cite{Kiuhf} Kirchberg proved that the full $C^*$-algebra
    of the free group $\F_n$ with $n>1$ (or $n=\infty$) generators,
    which is notoriously non-nuclear, also has the LP. 
    It follows that separable unital $C^*$-algebras with the LP
    are just the quotients of $C^*(\F_\infty)$ for which the quotient map
    admits a unital completely positive (u.c.p. in short)  lifting.
    More generally, as observed by Boca in \cite{Boca2}, it follows from the latter fact that 
    the LP is stable by unital (maximal) free products. Indeed, 
   it is an immediate consequence of  Boca's theorem in \cite{Boca}  (see
    \cite{DK} for a recent simpler proof)  that the free product of
    a family of unital $*$-homomorphisms that are liftable by u.c.p.
    maps is also liftable by a u.c.p. map.
    However it is well known that the reduced
    $C^*$-algebra of  $\F_n$ fails the LP. 
    
    Our main result is a 
    very simple (functorial) characterization of the LP in terms of maximal tensor products
    (namely \eqref{pty''} below) that seems to have been overlooked by previous authors.
    In \cite{Kir} Kirchberg 
    gave a tensor product characterizarion of the local version of the LP 
    called the LLP. He showed that a $C^*$-algebra $A$ has the LLP
    if and only if $B(\ell_2)  \otimes_{\max} A= B(\ell_2)  \otimes_{\min} A$.
    He then went on to conjecture that for separable
    $C^*$-algebras the LLP implies the LP, and he showed that a negative solution 
  would  imply a negative answer  for the Connes embedding 
    problem (see Remark \ref{glolo} for more information).
    We hope that our new criterion for the LP will help
    to answer this question whether the LLP implies the LP.
    More specifically, we believe that a modification of the construction in \cite{P7}    might lead to a counterexample.
    
 Let   $(D_i)_{i\in I}$ be a   family 
 of $C^*$-algebras. We will denote by $\ell_\infty(\{D_i\mid \in I\})$ or simply by
 $\ell_\infty(\{D_i\})$ the $C^*$-algebra formed of   the   families
    $d=(d_i)_{i\in I}$ in $\prod_{i\in I} D_i$ such that $\sup\nl_{i\in I} \|d_i\|<\infty$, equipped
    with the norm $d\mapsto \sup\nl_{i\in I}\|d_i\|$.
    \\{Consider the following property of a } $C^*${-algebra }A :
 \begin{equation}\label{pty}
 \text{For any family }(D_i)_{i\in I}
\text{ of }C^*\text{-algebras 
 and any }t\in \ell_\infty(\{D_i\}) \otimes A  
\text{ we have }
\end{equation}
 \begin{equation}\label{pty'}
 \|t\|_{\ell_\infty(\{D_i\}) \otimes_{\max} A} \le \sup_{i\in I} \|t_i\|_{ D_i \otimes_{\max} A} ,  \end{equation}
 where $t_i= (p_i\otimes Id_A)(t)$ with $p_i: \ell_\infty(\{D_i\}) \to D_i$ denoting the $i$-th coordinate projection.\\
Of course, for any $i\in I$ we have $\|p_i\otimes Id_A: {\ell_\infty(\{D_i\}) \otimes_{\max} A} \to D_i \otimes_{\max} A\|\le 1$, and hence we have a natural
contractive $*$-homomorphism
 \begin{equation}\label{12/11}
   {\ell_\infty(\{D_i\}) \otimes_{\max} A}\to {\ell_\infty(\{D_i \otimes_{\max} A\}) } .\end{equation}
Thus \eqref{pty'} means that
 we have a natural isometric embedding
   \begin{equation}\label{pty''}
   {\ell_\infty(\{D_i\}) \otimes_{\max} A}\subset {\ell_\infty(\{D_i \otimes_{\max} A\}) } .\end{equation}
 More precisely ${\ell_\infty(\{D_i\}) \otimes_{\max} A}$ can be identified with the closure
 of ${\ell_\infty(\{D_i\}) \otimes  A}$ (algebraic tensor product)
 in ${\ell_\infty(\{D_i \otimes_{\max} A\}) } $.

Let us denote
$$\C= C^*(\F_\infty),$$
the full (or ``maximal") $C^*$-algebra of the free group $\F_\infty$
with countably infinitely many generators.

 Using the description of the norm in $D \otimes_{\max} \ell_1^n   $ (for an arbitrary $C^*$-algebra $D$), 
 with $\ell_1^n   \subset \C$ identified as usual with the span of $n$ free generators,
 it is easy to check that $\C$ has this property \eqref{pty} (see 
 Lemma \ref{lpty} below).
 In fact, as a consequence,  any unital separable $A$ with LP has this property.
 Our main result is that conversely this characterizes the LP.
 
 The fact that \eqref{pty} implies the LP contains many   previously known lifting theorems, for instance the Choi-Effros one.
 It also gives a new proof of the LP for  $C^*(\F_\infty)$.

 The key step will be  a new form of the reflexivity principle.
  Consider for $E\subset A$ finite dimensional  
   the normed space $MB(E,C)$
 defined below in \S \ref{var} as   formed of the $\max\to\max$-tensorizing maps into another $C^*$-algebra $C$.
 We will show that if (and only if) $A$ has the property \eqref{pty} then the  natural map
 $$MB(E,C^{**}) \subset
MB(E,C)^{**}$$ is contractive for any  finite dimensional $E\subset A$.
As a consequence it follows that any $u\in MB(E,C)$
admits an extension in $MB(A,C)$ with the same $MB$-norm up to $\vp>0$.
From this extension property
we deduce a lifting one:
if  $A$ assumed unital and separable satisfies \eqref{pty},   any unital c.p. map $u: A \to C/\cl I$ ($C/\cl I$ any quotient)
  admits a unital c.p. lifting.   In other words
  this says that \eqref{pty} implies the LP.

  \begin{rem} Let $(D_i)_{i\in I}$  be a family  
 of $C^*$-algebras. 
 Consider the algebraic tensor product
 $\ell_\infty(\{D_i\}) \otimes A  $.
We will use the natural embedding
  $$\ell_\infty(\{D_i\}) \otimes A  \subset  \prod\nl_{i\in I} (D_i \otimes A).$$
  Any $t\in \ell_\infty(\{D_i\}) \otimes A  $
  can be identified with a bounded family
$(t_i)$ with $t_i\in D_i \otimes E$ for some f.d. subspace $E\subset A$.
 Thus the correspondence $t\mapsto (t_i)$
gives us a canonical {\it linear }
embedding  
  \begin{equation}  \label{13/8} \ell_\infty(\{D_i\}) \otimes A   \subset \ell_\infty(\{D_i\otimes_{\max} A  \}).\end{equation} 
  The property in   \eqref{pty} is equivalent
to the assertion that this map
  is an isometric embedding
  when  $\ell_\infty(\{D_i\}) \otimes A $ is equipped with the maximal 
  $C^*$-norm.
  \\
  More precisely, for any f.d. subspace $E\subset A$
we have a canonical linear isomorphism
$$\ell_\infty(\{D_i\}) \otimes E \simeq \ell_\infty(\{D_i \otimes E \}).$$
By convention, for any $C^*$-algebra $D$ let us denote by $D\otimes_{\max} E$
the normed space obtained by equipping $D\otimes E$ with the norm induced on it
by $D \otimes_{\max} A$.
Then the property in \eqref{pty} is equivalent
to the assertion that for any f.d. $E\subset A$
we have an isometric isomorphism
 \begin{equation}  \label{7/9} \ell_\infty(\{D_i\}) \otimes_{\max} E \simeq \ell_\infty(\{D_i \otimes_{\max} E \}).\end{equation}
\end{rem}

  \n{\bf Notation.} Let $D,A$ be $C^*$-algebras and let $E\subset A$ be a subspace.
  We will denote
  $$(D\otimes E)^+_{\max} =(D\otimes_{\max} A)_+ \cap (D\otimes E) 
   $$
   
  \begin{thm}\label{L1} Let $A$ be a separable $C^*$-algebra.
  The following are equivalent:\\
 \item{\rm (i)} The algebra  $A$ has the lifting property (LP).\\
  \item{\rm (ii)} For any family $(D_i)_{i\in I}$ 
 of $C^*$-algebras
 and any $t\in \ell_\infty(\{D_i\}) \otimes A  $
 we have
 $$\|t\|_{\ell_\infty(\{D_i\}) \otimes_{\max} A} \le \sup_{i\in I} \|t_i\|_{ D_i \otimes_{\max} A}  .$$
 In other words, the natural $*$-homomorphism \eqref{12/11}
 is isometric.\\
 \item{\rm (ii)}$_+$ For any family $(D_i)_{i\in I}$ 
 of $C^*$-algebras and any $t\in \ell_\infty(\{D_i\}) \otimes A$,
 the following implication holds
 $$\forall i\in I \ t_i \in (D_i \otimes A) ^+_{\max}
 \Rightarrow t\in  (\ell_\infty(\{D_i\}) \otimes A) ^+_{\max} .$$
 \end{thm}
We   prove Theorem \ref{L1} in \S \ref{pf}.
 
  \begin{rem}[On the nuclear case] 
   If one replaces everywhere $\max$ by $\min$ in (ii) in Theorem \ref{L1},
   then the property clearly holds for \emph{all} $C^*$-algebras.
   Thus Theorem \ref{L1} implies as a corollary the Choi-Effros lifting theorem from \cite{[CE5]},
   which asserts that nuclear $C^*$-algebras have the LP.
    \end{rem}
 \begin{rem} It is easy to see that
 if $A$ is the direct sum of finitely many $C^*$-algebras satisfying
  \eqref{pty} then $A$ also does.   
 \end{rem}
  \begin{rem}\label{rty0} Let $A$ be a $C^*$-algebra.
  It is classical (see e.g. \cite[Prop. 7.19]{P6}) that for any (self-adjoint, closed and two sided) ideal $\cl I\subset A$ and any $C^*$-algebra $D$
  we have an isometric embedding $D \otimes_{\max} \cl I
  \subset   D \otimes_{\max} A$. 
 This shows that if $A$ satisfies \eqref{pty} then 
 any ideal $\cl I$ in $A$ also does. \\Similarly, for any ideal
 $\cl I\subset D$ we have an isometric embedding $ \cl I \otimes_{\max} A
  \subset   D \otimes_{\max} A$. 
 Thus if  \eqref{pty} holds for a given family $(D_i)$ it also holds for
 any family $(\cl I_i)$  where each $\cl I_i$ is an ideal in $D_i$.
  Since any $D$ is an ideal in its unitization,
  one deduces from this that if \eqref{pty} holds for any family $(D_i)$
  of \emph{unital} $C^*$-algebras, then it holds for any family.
  One also sees that $A$ satisfies \eqref{pty} if and only if its unitization satisfies it.
 \end{rem}
 \begin{rem}\label{rty}
 We will use the fact due to Kirchberg \cite{Kir} that 
 for any $t\in D_i \otimes A$ there is a separable  $C^*$-subalgebra
 $\Delta_i\subset D_i$ such that $t\in \Delta_i \otimes A$ and 
 $$\|t\|_{\Delta_i \otimes_{\max} A} =\|t\|_{D_i \otimes_{\max} A} .$$
Indeed,  by \cite[Lemma 7.23]{P6} 
  for any $\vp>0$ there is a separable   
 $D^\vp_i \subset D_i$  such that $t\in D^\vp_i \otimes A$ and 
 $\|t\|_{D^\vp_i  \otimes_{\max} A} \le (1+\vp) \|t\|_{D_i \otimes_{\max} A} $.
 This implies that the $C^*$-algebra $\Delta_i$ generated by
 $\{D^\vp_i \mid \vp=1/n, n\ge 1 \}$ has the announced property.
 Using this fact in
  the property \eqref{pty} we may assume that all
  the $D_i$'s are separable (and unital by the preceding remark).
   \end{rem}
  \begin{rem}\label{rtyb}
 Let $C/ \cl I$ be a quotient $C^*$-algebra and let $A$ be another $C^*$-algebra.
 It is well known (see e.g. \cite[Prop. 7.15]{P6})
 that
   \begin{equation}  \label{25/8}
   (C/ \cl I)   \otimes_{\max} A= (C    \otimes_{\max} A )/    (    \cl I \otimes_{\max} A) .\end{equation}
   Equivalently, $A \otimes_{\max} (C/ \cl I)  = (A   \otimes_{\max} C)/    (   A \otimes_{\max}  \cl I)$.\\
   Moreover, for any f.d. subspace $E\subset A$ we have
   (for a detailed proof see e.g. \cite[Lem. 4.26]{P6})
    \begin{equation}  \label{7/9/1}
   (C/ \cl I)   \otimes_{\max} E= (C    \otimes_{\max} E )/    (    \cl I \otimes_{\max} E) .\end{equation}
   Moreover, the closed unit ball of $(C    \otimes_{\max} E )/    (    \cl I \otimes_{\max} E) $
   coincides with the image under the quotient map
   of the closed unit ball of  $ C    \otimes_{\max} E $. This known fact can be checked just like for the min-norm
   in \cite[Lem. 7.44]{P6}.\\
 Since any separable unital $C^*$-algebra $D$ can be viewed as a quotient of $\C$,
 so that say $D= \C/ \cl I$ we have an isomorphism
   \begin{equation}  \label{25/8/4} D   \otimes_{\max} A= (\C    \otimes_{\max} A )/    (    \cl I \otimes_{\max} A) .\end{equation}
   \end{rem}
  Using   the Remarks   \ref{rty} and \ref{rtyb} one  obtains:
     
    \begin{pro}\label{rty1} To verify the property \eqref{pty} we may assume without loss of generality that
  $D_i=\C$ for any $i\in I$. 
   \end{pro}
   
     \begin{lem}\label{ety}  Let $(D_i)_{i\in I}$ be a family of $C^*$-algebras.
     Then for any $C^*$-algebra $C$ we have  natural isometric embeddings
     \begin{equation}\label{30/8/1}  \ell_\infty(\{D_i \otimes_{\max} C\}) \subset  \ell_\infty(\{D^{**}_i\  \otimes_{\max} C\}).\end{equation}
      \begin{equation}  \label{30/8}\ell_\infty(\{D_i\}) \otimes_{\max} C \subset \ell_\infty(\{D^{**}_i\}) \otimes_{\max} C.\end{equation}
        \end{lem}
        
   \begin{proof}  
  We will use the classical fact that   \eqref{30/8} (or \eqref{30/8/1}) holds when $(D_i)_{i\in I}$ is reduced to a single
  element (see Remark \ref{ety'} or \cite[Prop. 7.26]{P6}), which means that we have for each $i\in I$ 
  a natural isometric embedding
  $D_i \otimes_{\max} C \subset  D^{**}_i\  \otimes_{\max} C$.
  From this \eqref{30/8/1} is immediate.\\
  To check \eqref{30/8}, since any unital separable $C$ is a quotient of $\C$, we may assume by Remark \ref{rtyb}
  that  $C=\C$ (see \cite[Th. 7.29]{P6} for details).
  Now  since $\C$ satisfies   \eqref{pty} (or equivalently the LP) we have isometric embeddings
  $$\ell_\infty(\{D _i\}) \otimes_{\max} \C \subset  \ell_\infty(\{D_i \otimes_{\max} \C\}) \text{  and  }
  \ell_\infty(\{D^{**}_i\}) \otimes_{\max} \C \subset \ell_\infty(\{D^{**}_i \otimes_{\max} \C\}) .$$ 
  Thus   \eqref{30/8} follows from \eqref{30/8/1} for $C=\C$, and hence in general.
 \end{proof}
      \begin{pro} To verify the property \eqref{pty} we may assume without loss of generality 
      that
  $D_i$ is  the bidual of a $C^*$-algebra for any $i\in I$. A fortiori we may assume 
   that
  $D_i$ is a von Neumann algebra for any $i\in I$.
   \end{pro}
   \begin{proof} This is an immediate consequence of Lemma \ref{ety}.
   \end{proof}
    
      In \cite{Kiuhf}, where Kirchberg shows that the $C^*$-algebra $\C=C^*(\F_\infty)$
      has the  LP, he also states 
      that if a $C^*$-algebra $C$ has the LP
then
for any von Neumann algebra $M$ 
            the nor-norm coincides on $M \otimes C$ (or on $C \otimes M$) with the max-norm.
      We will show that the converse also holds.
      The nor-norm of an element $t\in M\otimes C$ ($C$ any $C^*$-algebra) was defined by Effros and Lance 
      \cite{EL} as
            $$\|t\|_{\rm nor} =\sup\{ \| \sigma \cdot\pi (t)\|\}$$
            where the sup runs over all $H$ and all commuting pairs of representations          
            $\sigma: M \to B(H)$, $\pi: C \to B(H)$ with the restriction that $\sigma$ is normal on $M$. Here
            $\sigma \cdot\pi:  M \otimes C \to B(H)$ is the $*$-homomorphism defined by $\sigma \cdot\pi(m \otimes c)= \sigma(m) \pi( c)$ ($m\in M$, $c\in C$).
            The norm $\|\ \|_{\rm nor}$  is a $C^*$-norm and
              $M \otimes_{\rm nor} C$ is defined as the completion of $M\otimes  C$
            relative to this norm.
           One can formulate a similar definition for $C\otimes_{\rm nor} M$.
           See \cite[p. 162]{P6} for more on this.
          \\ We will invoke the following elementary fact (we include its  proof
          for lack of 
            a reference).
           
           \begin{lem}\label{19/11} Let $D$ be another $C^*$-algebra.
           Then for any c.c.p. map $u:C\to D$ the mapping
           $Id_M \otimes  u: M \otimes C \to M \otimes D$ extends to a 
           contractive (and
           c.p.) map from $M \otimes_{\rm nor} C$
           to
           $M \otimes_{\rm nor} D$.
           \end{lem}
            \begin{proof} Let $S_{\max}$ (resp. $S_{\rm nor}$) denote the set of states
            on $M \otimes_{\rm max} C$ (resp. $M \otimes_{\rm nor} C$).
            An element of  $S_{\max}$ can be identified
          with a bilinear form $f: M \times  C \to \bb C$ of norm $\le 1$ that is c.p.
            in the sense that $\sum f(x_{ij}, y_{ij}) \ge 0$
            for all $n$ and all  $[x_{ij}]\in M_n(M)_+ $ and $[y_{ij}]\in M_n(C)_+ $. The set $S_{\rm nor}$  
            corresponds to the  forms $f\in S_{\max}$      that are
            normal in the first variable. See \cite{EL} or \cite[\S 4.5]{P6} for more information.
     Let  $u:C\to D$ be a c.c.p. map.
            Clearly,  for any  $f\in S_{\rm nor}$ the form $(x,y)\mapsto f(x, u(y))$  is
            still  in $S_{\rm nor}$. Since $\|t\|_{\rm nor} \approx \sup_{f\in S_{\rm nor}} |f(t)|$
            for any $t\in M \otimes  C$, the mapping $Id_M \otimes u$
            (uniquely) extends  to a bounded linear map     $u_M:M \otimes_{\rm nor} C\to
           M \otimes_{\rm nor} D$. 
           Since $u$ is c.p. we have $u_M(t^*t)\in  (M \otimes C)_+ :={\rm span}\{s^*s\mid s\in M \otimes C \}$ for any $t\in M \otimes  C$,
           and hence by density $u_M(t^*t)\in  (M \otimes_{\rm nor} C)_+$ for any $t\in M \otimes_{\rm nor}  C$.
           This means that $u_M$ is positive. Replacing $M$ by $M_n(M)$  shows that $u_M$ is c.p. Since $\|t\|_{\rm nor} = \sup_{f\in S_{\rm nor}} |f(t)|$ when $t\ge 0$, we have $\|u_M\|\le 1$.
       \end{proof}
           We will invoke the following simple elementary fact.
           \begin{lem}\label{nty} Let $A$ be a   $C^*$-algebra.
           For any family $(D_i)_{i\in I}$ 
 of $C^*$-algebras we have  a natural isometric embedding 
            $$\ell_\infty(\{D^{**}_i \}) \otimes_{\rm nor} A\subset \ell_\infty(\{D^{**}_i \otimes_{\rm nor} A\}) .$$          
           \end{lem}
            \begin{proof}   
            Let $D=c_0(\{D_i\})$ so that $D^{**}=\ell_\infty(\{D^{**}_i\}) $.              Consider $t\in D^{**}\otimes  A$.  Let $\pi: A \to B(H)$, $\sigma: D^{**} \to \pi(A)'$
            be commuting non-degenerate $*$-homomorphisms with $\sigma$ \emph{normal}.  Thus $\sigma$   is the canonical weak* to weak* continuous extension
            of $\sigma_{|D}: D \to \pi(A)'$. It will be notationally convenient to
        view $D_i$ as a   subalgebra of $D$.
            We observe that  $\sigma_{|D}$ is the direct sum of representations of the $D_i$'s.
            Let  $\sigma_i: D_i  \to \pi(A)'$ be defined by
             $ \sigma_i(x) = \sigma(x) $ (recall   $D_i\subset D$). 
            There is an orthogonal decomposition of $H$ of the form
            $Id_H= \sum_{i\in I} p_i$ with $p_i\in \sigma(D) \subset \pi(A)'$ such that
            for any $c=(c_i)\in D$ we have
            $\sigma(c)= \text{norm sense}\sum_{i\in I}   \sigma_i(c_i) $ and also
              $\sigma_i(x)=p_i \sigma_i(x)=\sigma_i(x)p_i $ for any $x\in D_i$.
            Let $\ddot\sigma_i: D_i^{**} \to \pi(A)'$
            be the weak* to weak* continuous extension of $\sigma_i$. Since $\sigma$ is normal,
            we have then for any $c''=(c_i'')\in \ell_\infty(\{D^{**}_i\}) $
            $$\sigma(c'') =\text{weak* sense}   \sum\nl_{i\in I} \ddot\sigma_i(c_i'').$$
            This means $ \sigma\simeq \oplus  \ddot\sigma_i$.
            Therefore for any $t\in D^{**} \otimes  A= \ell_\infty(\{D^{**}_i\}) \otimes  A$ we have
            $$\| \sigma \cdot\pi (t)\|= \sup\nl_{i\in I} \| \ddot\sigma_i \cdot \pi(t_i) \|\le \sup\nl_{i\in I} \|t_i\|_{D^{**}_i \otimes_{\rm nor} A}.$$
            Taking the sup over all the above specified pairs $(\sigma,\pi)$,
            we obtain $\|t\|_{\rm nor} \le \sup\nl_{i\in I}\|t_i\|_{D^{**}_i \otimes_{\rm nor} A},$
whence (since the converse is obvious) a natural isometric embedding 
            $$D^{**}\otimes_{\rm nor} A\subset \ell_\infty(\{D^{**}_i \otimes_{\rm nor} A\}) .$$
 This completes the proof. \end{proof}

           The next statement is now an easy consequence of Theorem \ref{L1}.
           \begin{thm}\label{7/9/3}
           \label{LP2} Let $A$ be a separable $C^*$-algebra.
  The following are equivalent:
 \item{\rm (i)} The algebra  $A$ has the lifting property (LP).
  \item{\rm (i)'} The algebra  $A$ satisfies \eqref{pty}.
   \item{\rm (ii)} For any von Neumann algebra $M$ we have
    \begin{equation}  \label{27/9}
     M\otimes_{\rm nor} A=M\otimes_{\max} A \ \text{ (or equivalently }A\otimes_{\rm nor} M=A\otimes_{\max} M).\end{equation}    
   \item{\rm (ii)'} For any $C^*$-algebra $D$ we have 
   $$D^{**}\otimes_{\rm nor} A=D^{**}\otimes_{\max} A \ \text{ (or equivalently }A\otimes_{\rm nor} D^{**}=A\otimes_{\max} D^{**}).$$
           \end{thm}
            \begin{proof}  
         The equivalence  (i) $\Leftrightarrow$ (i)'  duplicates for convenience part of Theorem \ref{L1}.\\
            (i) $\Rightarrow$ (ii) boils down to Kirchberg's result
            from \cite{Kiuhf} 
            that $\C$ satisfies (ii),
            for which a simpler proof was given in \cite{Pjot} (see also 
            \cite[Th. 9.10]{P6}).
            Once this is known, if $A$ has the LP then $A$ itself satisfies (ii). 
            Indeed, we may assume $A=\C/\cl I$ and by Lemma \ref{19/11}
             if $r: A \to \C$ is a c.c.p.   lifting 
            then 
            $r_M: M\otimes _{\rm nor} A$ to $M\otimes _{\rm nor} \C=M\otimes _{\rm max} \C$
             is   contractive, from which \eqref{27/9} follows. A priori
                this uses the separability of $A$ but  we will give a direct proof of (i)'  $\Rightarrow$ (ii) valid in the non-separable case in \S \ref{non-sep}.
 \\              
     (ii) $\Rightarrow$  (ii)' is trivial.    Assume (ii)' (with $A$ possibly non-separable). By Lemma \ref{nty} we have
                          a natural isometric embedding 
            $$\ell_\infty(\{D^{**}_i\})\otimes_{\max} A\subset \ell_\infty(\{D^{**}_i \otimes_{\max} A\}) .$$
              By
            \eqref{30/8/1} and \eqref{30/8} we must have \eqref{pty}, which means 
            that (i)' holds. Thus (ii)' $\Rightarrow$ (i)'.
          \end{proof}
   
   We will prove a variant of the preceding theorem in terms of ultraproducts
   in \S \ref{uuty}.
     \begin{rem} It is known (see \cite[Th. 8.22]{P6}) that
 we always have an isometric natural embedding
     $$   D^{**} \otimes_{\rm bin} A^{**}\subset (D \otimes_{\max} A)^{**} .$$
     Therefore, for arbitrary $D$ and $A$, we have an isometric natural embedding
      \begin{equation} \label{15/11}D^{**} \otimes_{\rm nor} A \subset  (D \otimes_{\max} A)^{**}.\end{equation}
     Thus (ii)' in Theorem \ref{LP2} can be reformulated
     as saying that for any $D$ we have an isometric natural embedding
     \begin{equation} \label{16/11} D^{**} \otimes_{\max} A \subset  (D \otimes_{\max} A)^{**}.\end{equation}
     This is the analogue for the max-tensor product of 
     Archbold and Batty's property $C'$ from \cite{[AB]}, which is closely related to the local reflexivity  of \cite{EH} (see \cite[p. 310]{P4}
     or \cite[Rem. 8.34]{P6} for more information on this topic).
     However, what matters here is   the injectivity of the $*$-homomorphism in \eqref{16/11}. Its continuity is guaranteed by \eqref{15/11}. In sharp contrast
     for property $C'$ continuity is what matters while injectivity is
     automatic since the min-tensor product is injective.
     
     \end{rem}
      \begin{rem}\label{ety'} Let $A$ be a $C^*$-algebra.
      We already used in Lemma \ref{ety} that fact that for any other $C^*$-algebra
      $D$  we have an \emph{isometric} natural morphism
      $D\otimes_{\max} A \to D\otimes_{\max} A^{**}$.
      Similarly, for any von Neumann algebra $M$, the analogous
      morphism
      $M\otimes_{\rm nor} A \to M\otimes_{\rm nor} A^{**}$
      is \emph{isometric}.  This follows from the basic observation that
      if $\sigma: M\to B(H)$ and $\pi:A \to B(H)$ are representations 
      with commmuting ranges, then the canonical normal
      extension $\ddot \pi$ of $\pi$ to $A^{**}$
    takes values in $\sigma(M)'$.
      \end{rem}
   \begin{rem}[On the LLP]  Following Kirchberg \cite{Kir},
   a unital $C^*$-algebra $A$ is said to have the  local lifting property (LLP) if 
   for any unital c.p. map $u: A \to C/\cl I$ ($C/\cl I$ being any quotient  of a  unital $C^*$-algebra $C$)
   the map $u$ is ``locally liftable" in the following sense: for any f.d. operator system
   $E\subset A$ the restriction $u_{|E}: E \to C/\cl I$ admits a u.c.p. lifting
   $u^E: E \to C$. The \emph{crucial difference} between ``local and global" is that 
   a priori $u_{|E}$ \emph{does not} extend to a u.c.p. map on the whole of $A$.
   If $A$ is not unital we say that it has the LLP if its unitization does.
   Equivalently a general $C^*$-algebra $A$  has the LLP
   if  for any c.c.p. map $u: A \to C/\cl I$ ($C/\cl I$ being any quotient  of a $C^*$-algebra $C$) and any f.d. subspace  
   $E\subset A$ the restriction $u_{|E}: E \to C/\cl I$ admits a   lifting
   $u^E: E \to C$ such that $\|u^E\|_{cb} \le 1$ (see \cite[Th. 9.38]{P6}).
\\
       Kirchberg \cite{Kir} proved that $A$ has the LLP if and only if 
   $B(\ell_2) \otimes_{\min} A= B(\ell_2) \otimes_{\max} A$ or equivalently 
    if and only if $\ell_\infty(\{M_n\mid n\ge 1\}) \otimes_{\min} A= \ell_\infty(\{M_n\mid n\ge 1\})  \otimes_{\max} A$.
    Using this it is easy to check, taking $\{D_i\}=\{M_n\mid n\ge 1\}$,  that \eqref{pty}
    implies the LLP.
   \end{rem}
   \begin{rem}[Global versus local]\label{glolo}
   Clearly the LP implies the LLP.  Kirchberg observed in \cite{Kir} that 
  if his conjecture\footnote{\bf according to a recent paper entitled
MIP* = RE posted on arxiv in Jan. 2020 by
 Ji, Natarajan, Vidick,  Wright, and  Yuen  this conjecture is not correct} that $\C \otimes_{\min} A= \C \otimes_{\max} A $
  for any $A$ with LLP (or equivalently just for $A=\C$)   is correct
  then conversely the LLP implies the LP for separable $C^*$-algebras.
  \\
   Note that if $\C \otimes_{\min} \C= \C \otimes_{\max} \C $, then
   since $\C$ satisfies \eqref{pty} we have $\ell_\infty(\C) \otimes_{\min} \C= \ell_\infty(\C) \otimes_{\max} \C $,
   and  hence for any  $A$ with LLP also
   \begin{equation}  \label{4/9}
   \ell_\infty(\C) \otimes_{\min} A= \ell_\infty(\C) \otimes_{\max} A .\end{equation} 
      Now  \eqref{4/9}  obviously implies (recalling Proposition \ref{rty1}) the property in \eqref{pty}, which by Theorem \ref{L1}
    implies the LP in the separable case, whence another viewpoint on Kirchberg's observation.

    Kirchberg showed in \cite{Kir} that a $C^*$-algebra $D$ has the 
    weak expectation property (WEP) (for which 
    we refer the reader to  \cite[p. 188]{P6}) if and only if $D \otimes_{\max} \C=D \otimes_{\min} \C$.
    Thus his 
      conjecture is equivalent to the assertion that $\C$ has 
   the WEP or that any $C^*$-algebra $D$ is a    quotient of a WEP $C^*$-algebra. Such $D$s are called QWEP. He also showed that it is  equivalent
    to the Connes embedding problem (see \cite[p. 291]{P6}).
    
    Note that if all the $D_i$ are WEP, the fact that $\C$
    satisfies \eqref{pty} tells us that
    $   {\ell_\infty(\{D_i\}) \otimes_{\max} \C}\to {\ell_\infty(\{D_i \otimes_{\min} \C\}) }$
    is isometric, and hence that $   {\ell_\infty(\{D_i\}) \otimes_{\max} \C}=
       {\ell_\infty(\{D_i\}) \otimes_{\min} \C}$. Thus it follows that
    $\ell_\infty(\{D_i\})$ is WEP.

    Yet another way to look at the problem whether
   the LLP implies the LP for separable $C^*$-algebras is through the observation (pointed out by a referee) that
    $A$ has the LLP if and only if $   {\ell_\infty(\{D_i\}) \otimes_{\max} A}\to {\ell_\infty(\{D_i \otimes_{\max} A\}) }$ is isometric (or equivalently injective)
    for any family $(D_i)$ of QWEP $C^*$-algebras.
    This is easily checked using \eqref{25/8/4}. If Kirchberg's conjecture was correct, the LLP of $A$   would imply the same for arbitrary $D_i$ and hence by Theorem \ref{L1} this would imply that $A$ has the LP in the separable case.
\end{rem}

   \begin{rem}[Counterexamples to LP] For the  full $C^*$-algebra $C^*(G)$ of a discrete group $G$,
   the LP holds both when $G$ is   amenable   and when $G$ is a  free group. It is thus not easy to find
   counterexamples, but  the existence of $C^*(G)$'s failing LP has been proved using property (T)
   groups by Ozawa in \cite{Ozpams}. Later on, Thom \cite{Thom} gave an explicit example of $G$
   for which  $C^*(G)$ fails the LLP. More recently, Ioana, Spaas and Wiersma \cite{[ISW]}
   proved that $ SL_n(\Z)$ for $n\ge 3$ and many other similar property (T) groups fail the LLP. See \S \ref{ill} for more on this theme.
   \end{rem}
   
   {\bf Abbreviations and notation} As is customary
   we abbreviate completely positive by c.p.
   completely bounded by c.b. unital completely positive by u.c.p.
   and contractive completely positive by c.c.p.
   Analogously we will abbreviate maximally bounded, maximally positive and unital maximally positive
   respectively by m.b. m.p. and u.m.p.
     We also use f.d. for finite dimensional.
   We denote by $B_E$ the closed unit ball of a normed space $E$,
   and by $Id_E: E \to E$ the identity operator on $E$.
   We denote by $E\otimes F$ the \emph{algebraic} tensor product of two vector
   spaces.
   Lastly, by an ideal in a $C^*$-algebra we implicitly mean a two-sided, closed and self-adjoint ideal.
   
  \section{Maximally bounded   and maximally positive maps}\label{var}
  
 Let $E\subset A$ be an operator space
  sitting in a $C^*$-algebra $A$.
  Let $D$ be another $C^*$-algebra. 
  Recall we denote (abusively) by $D  \otimes_{\max }  E$
  the closure of $D  \otimes   E$ in $D\otimes_{\max}  A$, and we denote
  by $\| \  \|_{\max}$ the norm induced
  on  $D  \otimes_{\max }  E$ by $D  \otimes_{\max }  A$.
  We define similarly $E  \otimes_{\max }  D$. 
  We should emphasize that
  $D  \otimes_{\max }  E$ (or $E  \otimes_{\max }  D$) depends on $A$ and on the embedding
  $E\subset A$, but there will be no risk of confusion.
  Of course we could also define $E\otimes_{\max} F\subset A\otimes_{\max} D$ for a subspace $F\subset D$
  but we will not go that far.
   Let $C$ be  another $C^*$-algebra.
  
  We will denote by ${MB}(E,C)$ the set of maps $u: E \to C$
  such that for any $C^*$-algebra $D$, the map $Id_D \otimes
 u: D \otimes E \to  D \otimes C$  extends to a bounded
  mapping $u_D: D  \otimes_{\max }  E\to D  \otimes_{\max }  C$,
  and moreover such that the bound on $\|u_D\|$ is uniform over all $D$'s.
  We  call such maps ``maximally bounded" or ``max-bounded"
  or simply m.b. for short. (They are called $(\max\to \max)$-tensorizing in \cite{P6}.)
  We denote
  $$\|u\|_{ mb}= \sup\{ \| 
 u_D: D \otimes_{\max } E \to  D \otimes_{\max } C\| \}$$
 where the sup runs over all possible $D$'s.\\
We have clearly $\|u_D\|_{mb} \le \|u\|_{ mb}$.

Moreover, if $A$ is unital, assuming that $E$ is an operator system 
let us denote by ${MP}(E,C)$ the subset formed of the   maps $u$
such that $Id_D \otimes
 u: D \otimes E \to  D \otimes C$ is positive  for any $D$, by which we mean
   that $$[Id_D \otimes
 u]  ( ( D \otimes E)^+_{\max}) \subset ( D \otimes C)^+_{\max}.$$
 Equivalently, $u_D$ is positive for all $D$'s.
 We call such maps ``maximally positive" (m.p.  in short).\\ 
 Replacing $D$ by $M_n(D)$ shows that the  latter maps will be automatically c.p.\\
 In passing, recall that for any c.p. $u:E \to C$ we have  $\|u\|_{cb}=\|u\|=\|u(1)\|$,
 and hence $\|u\|_{mb}=\|u(1)\|$ for any m.p. $u$.\\  
Obviously we have  $\|u\|_{ mb}=1$ if $E=A$ and $u$ is a  $*$-homomorphism.
 
 Clearly (taking $D=M_n$) we have ${MB}(E,C)\subset CB(E,C)$ and
  \begin{equation}  \label{23/8} \forall u \in {MB}(E,C)\quad 
  \|u\|_{cb} \le \|u\|_{ mb}.
  \end{equation}      
 \begin{pro}  
    A map $u: E \to C$ belongs to  ${MB}(E,C)$ if and only if $Id_{\C} \otimes u$ defines a bounded map
    from $ \C \otimes_{\max} E $ to  $ \C    \otimes_{\max} C$ 
    and we have
    $$  \|u\|_{ mb}= \|Id_{\C} \otimes u: \C \otimes_{\max} E \to  \C    \otimes_{\max} C\|.$$ 
   \end{pro}
    \begin{proof} This follows easily from \eqref{25/8/4}.
     \end{proof}
     \begin{cor} If $C$ has the WEP then any
     c.b. map $u: E \to C$ is in ${MB}(E,C)$ and
      $\|u\|_{cb} = \|u\|_{ mb}$.
  \end{cor}
   \begin{proof}  By the WEP of $C$, we have $\C \otimes_{\max} C=\C \otimes_{\min} C$ (see \cite[p. 380]{BO} or \cite[p. 188]{P6}), and hence
   $$\|Id_\C \otimes u: \C \otimes_{\max} E \to \C \otimes_{\max} C\|=
   \|Id_\C \otimes u: \C \otimes_{\max} E \to \C \otimes_{\min} C\|
   $$
   $$\le \|Id_\C \otimes u: \C \otimes_{\min} E \to \C \otimes_{\min} C\|\le \|u\|_{cb}.$$
     \end{proof}
\begin{rem}\label{rc} Let $u\in {MB}(E,C)$. Let $F\subset C$ be such that
$u(E)\subset F$ and let $B$ be another $C^*$-algebra. Then for any
 $v\in {MB}(F,B)$ the composition $vu: E \to B$ is in
 ${MB}(E,B)$ and $\|vu\|_{mb} \le  \|v\|_{mb}\|u\|_{mb}$.
 \end{rem}
 \begin{rem}\label{rw} When $E=A$  any   c.p. map
 $u: A \to C$  is in $ {MB}(A,C)$ and $\|u\|_{mb}=\|u\|$
 (see e.g.  \cite[p. 229]{P4} or \cite[Cor. 7.8]{P6}).
 However, this is no longer true in general when $E$ is merely an operator system in $A$. \emph{This distinction is important for the present work.}
 \end{rem} 
 
  \begin{rem}\label{wk} We will use the following well known elementary fact.
  Let $A,C$ be unital $C^*$-algebras. Let $E\subset A$ be an operator system.
  Let  $u: E \to C $ be a unital c.b. map.
  Then $u$ is c.p. if and only if $\|u\|_{cb}=1$ (see e.g. \cite{Pa2},  \cite[p. 24]{P4} or \cite[Th. 1.35]{P6}).
  \end{rem}
 
  \begin{rem}\label{rw1}  In the same situation, any \emph{unital} map $ u\in {MB}(E,C)$
  such that $\|u\|_{mb}=1$ must be   c.p. 
 Indeed,  
  \eqref{23/8} implies $\|u\|_{cb}=\|u(1)\|=1$,  and
  $u$ is  c.p.
 by the preceding remark.
 More precisely,   the same reasoning applied to the maps $u_D$ 
 (with $D$ unital) shows that $u$ is  m.p.
 \end{rem}
   \begin{rem}\label{rw2}  
   Let $u: E \to C$ with $E$ an operator system. Let $u_*:E \to C$ be defined by
   $u_*(x)=(u(x^*))^*$ ($x\in E$), so that $\|u_*\| =\|u\| $. We claim
   $\|u_*\|_{mb}=\|u\|_{mb}$ for any $u\in MB(E,C)$. Indeed, this is easy to check
   using $(u_D)_*=(u_*)_D$.\\
   A map $u: E \to C$ is called ``self-adjoint" if $u=u_*$.
 \end{rem}

The following important result  was shown to
the author by Kirchberg with permission to include it in \cite{P4}.
It also appears in \cite[Th. 7.6]{P6}. 

\begin{thm}[\cite{Ki6}]\label{L2} Let $A,C$ be $C^*$-algebras. Let
$i_C: C \to C ^{**}$ denote the inclusion map.
A map $u: A \to C$ is in  ${MB}(A,C)$
if and only if  $i_C u: A \to C^{**}$ is decomposable. Moreover
we have
$$\|u\|_{mb} = \|i_C u \|_{dec}.$$
 \end{thm}
 Recall that a map $u: A \to C$ is called decomposable if it is a linear combination of c.p. maps.
 See  \cite{Ha} or 
\cite[\S 6]{P6} for  more on decomposable maps and the definition of the dec-norm.

 Elaborating on Kirchberg's argument we included in 
 \cite{P4} and later again  in \cite[Th. 7.4]{P6}
 the following variant as an extension theorem.
 
 \begin{thm}\label{ext} Let $A,C$ be $C^*$-algebras.
 Let $E\subset A$ be a subspace. 
Consider a map $u: E \to C$  in ${MB}(E,C)$.
Then 
$$\|u\|_{mb} = \inf\| \tilde u  \|_{dec},$$
where the infimum runs over all maps $\tilde u : A \to C^{**}$
such that ${\tilde u}_{|E} =i_C u$. Moreover, the latter infimum is attained.
 \end{thm}
  \begin{rem}\label{r22} A fortiori, since $  \|  i_C u  \|_{dec}    \le \|  u  \|_{dec}$, we have
  $\|u\|_{mb} \le \|  u  \|_{dec}$.\\
  Moreover, for any $\tilde u : A \to C^{**}$ we have
  \begin{equation}\label{e22}
  \| \tilde u\|_{mb} =\| \tilde u  \|_{dec}.\end{equation}
  Indeed, since there is a c.c.p. projection
  $P: (C^{**})^{**} \to C^{**}$, we have 
    $\| \tilde u\|_{mb}  =\| i_{C^{**}}\tilde u\|_{dec} \ge
     \| P i_{C^{**}}\tilde u  \|_{dec}= \| \tilde u  \|_{dec}.$
  \end{rem} 
  \begin{rem} For any $C^*$-algebras $C,D$ the natural inclusion
  $D \otimes_{\max} C \subset D \otimes_{\max} C^{**}$ 
  is isometric (see Remark \ref{ety'}). Therefore,   for any $u: E \to C$ we have
  $$\|u\|_{mb} = \|i_C u\|_{mb} $$
  where $i_C: C \to C^{**}$ is the canonical inclusion.\\
  Moreover, if there is a projection $P: C^{**}\to C$ with $\|P\|_{dec}=1$
  (for instance if $P$ is c.c.p.) then
  $\| u\|_{dec} = \|i_C u\|_{dec} $.
  \end{rem}
 \begin{rem}[A refinement of Theorem \ref{ext}]
 Let $E\subset A$ and let $M\subset B(H)$ be a von Neumann algebra.
 Consider a map $u: E \to M$.
 Let us denote by $\hat u: M' \otimes E \to B(H)$ the linear map
 such that $\hat u(x'\otimes a) =x'u(a)$  ($x'\in M'$, $a\in E$).
 Then $$\|u\|_{mb} =\|\hat u : M' \otimes_{\max} E  \to B(H)\|_{cb}
 =  \inf\{\| \tilde u  \|_{dec} \mid \tilde u: A \to M,\  \tilde u_{|E}=u\}.$$
 If either  $M$ has infinite multiplicity
 or  $M'$ has a cyclic vector, then 
 $$\|\hat u : M' \otimes_{\max} E  \to B(H)\|_{cb}=\|\hat u : M' \otimes_{\max} E  \to B(H)\|.$$
 This is Th. 6.20 with Cor. 6.21 and Cor. 6.23 in \cite{P6}. \end{rem}
 
 \begin{rem}\label{trick} In the situation of the preceding remark,
 if $E$ is an operator system and if $\hat u : M' \otimes_{\max} E  \to B(H)$
 is c.p. (in particular if $u\in MP(E,M)$)
 then $u$ admits a c.p. extension $\tilde u :A \to M$
 with $\| \tilde u \|=\|u\|$ (in particular $\tilde u\in MP(A,M)$).
 This follows from Arveson's extension theorem for c.p. maps
 and the argument called ``the trick" in \cite[p. 87]{BO}.
 \end{rem}
 Let $E\subset A$ and $C$ be as in Theorem \ref{ext}. 
 To state our results in full generality, we
 introduce the space ${SB}(E,C)$ for which the unit ball
is
formed of all $u: E\to C$
satisfying an operator 
  analogue of \eqref{pty}.
  \begin{dfn}
  We call strongly maximally bounded the maps
  $u: E\to C$ such that, for any family $(D_i)_{i\in I}$, the map  $Id \otimes u$  defines 
   a uniformly bounded map
  from $\ell_\infty(\{D_i\}) \otimes  E $ 
 equipped with the norm induced by
 $\ell_\infty(\{D_i \otimes_{\max}  E\}) $ to $\ell_\infty(\{D_i\}) \otimes  C $
 equipped with the norm induced by
    $\ell_\infty(\{D_i\} )\otimes_{\max}  C$.
 We define 
 \begin{equation}\label{ptyo}\|u\|_{{sb}}= \sup\{\| (Id \otimes u)(t)\|_{\ell_\infty(\{D_i\}) \otimes_{\max} C} \mid  (t_i)_{i\in I} \in \ell_\infty(\{D_i\}) \otimes  E,\  \sup_{i\in I} \|t_i\|_{ D_i \otimes_{\max} C} \le 1\} ,\end{equation} 
 where the sup runs over all possible families $(D_i)_{i\in I}$ 
 of $C^*$-algebras.
 \end{dfn}
 
  \begin{dfn}\label{d23} Let $A,C$ be  $C^*$-algebras. Let $E\subset A$ be a subspace.
   A linear map $u: E\to C$ will be called maximally isometric (or $\max$-isometric in short)
   if for any $C^*$-algebra $D$ the associated map $Id_D \otimes u: D\otimes E\to D\otimes C$
   extends to an isometric map $D\otimes_{\max} E\to D\otimes_{\max} C$. 
   Equivalently, this means that $E\otimes_{\max} D\to C\otimes_{\max} D$ is isometric for all $C^*$-algebras $D$.
    \end{dfn}
These maps were called 
   $\max$-injective in \cite{P6}. We adopt here ``maximally  isometric"
   to emphasize the analogy with completely isometric.
 \begin{rem}\label{r27/11} For example, for any
$C^*$-algebra $C$ the canonical inclusion
$i_C : C \to C^{**}$ is $\max$-isometric (see 
Remark \ref{ety'} or \cite[Cor. 7.27]{P6}).  
Thus for any $D$ we have an isometric
$*$-homomorphism $D\otimes_{\max} C \to D\otimes_{\max} C^{**}$,
which implies
\begin{equation}\label{27/11}
(D\otimes C) \cap (D\otimes_{\max} C^{**})_+ \subset (D\otimes_{\max} C)_+ .
\end{equation}
Moreover,
the inclusion 
$\cl I\to C$ of an ideal, in particular the inclusion
$C\to C_1$ of $C$ into its unitization,
is $\max$-isometric (see e.g. \cite[Prop. 7.19]{P6}).
\end{rem}

\begin{rem}\label{rw13}  
We have $\|Id_A\|_{SB(A,A)}=1$ 
if (and only if) $A$ satisfies the property  \eqref{pty}.
In that case, by Remark \ref{rw1} for any 
$u\in CP(A,C)$
(where $C$ is an arbitrary $C^*$-algebra)
 we
 have
 \begin{equation}\label{13/11}
 \|u\|_{SB(A,C)}=\|u\|.
 \end{equation}
\end{rem}

\begin{rem}\label{r23} We will use the elementary fact (see \cite[Cor. 7.16]{P6} for a proof)
 that  $D \otimes_{\max} E \to D \otimes_{\max} C$
 is isometric for all $D$ (i.e. it is $\max$-isometric) if and only if
this holds for $D=\C$. In particular checking $D$ separable is enough.
\end{rem}

 As we will soon show (see Theorem \ref{lift}), maximally bounded maps admit
 maximally bounded liftings when \eqref{pty} holds.
 To tackle u.c.p. liftings we will need  a bit more work.
The following perturbation lemma is the m.p. analogue of  Th. 2.5 in \cite{EH}. 

\begin{lem}\label{l27/11}
Let $E\subset A$ be an $n$-dimensional operator system, $C$ a unital $C^*$-algebra.
Let  $0<\vp<1/2n$. For any self-adjoint unital map  $u: E \to C$  with $\|u\|_{mb}\le 1+\vp$,
there is a unital maximally positive 
(u.m.p. in short) map $v: E \to C$ such that $\|u-v\|_{mb}\le 8n\vp$.\\
Here  $u$ self-adjoint  means that $u(x)=u(x^*)^*$ for any $x\in E$.
\end{lem}
\begin{proof}[Sketch] The proof is essentially the same as that
of the corresponding statement for the c.b. norm and c.p. maps
appearing in \cite[Th. 2.5]{EH} and reproduced in \cite[Th. 2.28]{P6}.
However we have to use the ``m.p. order" instead of the c.p. one.
The only notable difference is  
that one should use Theorem \ref{ext} instead of the injectivity of $B(H)$.
Using the latter we can find a self-adjoint extension $\tilde u: A \to C^{**}$
with $\|  \tilde u\|_{dec}\le 1+\vp$. We can write
 $\tilde u= v_1-v_2$ with $v_1$ and $v_2$ in $MP(A,C^{**})$ such that
 $\|v_1+v_2\|\le 1+\vp$. Since $u(1)=1$ we have
 $\|  1 + 2 v_2(1)\|=\|(v_1+v_2)(1)\|\le 1+\vp$, and hence
 $\|v_2(1)\|\le \vp/2$. Arguing as in \cite[Th. 2.5]{EH} or \cite[Th. 2.28]{P6}
 there is $f\in E^*$ with $\|f\|\le 2n\vp$ such that
 the mapping $w_2: x\mapsto f(x)1-v_2(x)$ is in $MP(E,C^{**})$.
 We then set $v_0={v_1}_{|E}+w_2= u+  f(\cdot)1$.
 Note that $v_0(E)\subset C$. A priori $v_0\in MP(E,C^{**})$, but
  we actually have
 $v_0\in MP(E,C)$ by
 \eqref{27/11}.
 Moreover $\|v_0 -u\|_{mb} \le   2n\vp$.
 In particular $\|v_0(1) -1\|  \le   2n\vp <1$, which shows that $v_0(1)$
 is invertible and close to $1$ when $\vp$ is small.
 The rest of the proof is as in \cite[Th. 2.5]{EH} or \cite[Th. 2.28]{P6}.
\end{proof}

The following variant is immediate:
\begin{lem}\label{v26}
Let $E\subset A$ be an $n$-dimensional operator system, $C$ a unital $C^*$-algebra.
  For any self-adjoint map  $u: E \to C$
  such that $\|u(1)-1\|<\vp$ and  $\|u\|_{mb}\le 1+\vp$,
there is a unital map $v\in MP( E , C)$ such that $\|u-v\|_{mb}\le  f_E(\vp)$,
where $f_E$ is a function of $\vp\in (0,1)$ such that $\lim_{\vp\to 0} f_E(\vp)=0$.
\end{lem}

\begin{proof}[Sketch] Let $u'(\cdot)=u(1)^{-1/2} u(\cdot) u(1)^{-1/2} $.
Then $\|u-u'\| \le f'(\vp)$ with $\lim_{\vp\to 0} f'(\vp)=0$.
We may apply Lemma \ref{l27/11} to $u'$.
 \end{proof}

 \section{Arveson's principle}\label{arvp}
 
 To tackle global lifting problems a principle due to Arveson
 has proved very useful.
 It asserts roughly that in the separable case pointwise limits
 of suitably liftable maps are liftables.
    Its general  form
   can be stated as follows.
   Let $E$ be a separable operator space,
   $C$ unital $C^*$-algebra, $\cl I\subset C$ an ideal.
 
A bounded subset of $\cl F\subset  B(E,C)$ will be called admissible if
   for any pair $f,g$ in $\cl F$
   and any $\sigma \in C_+$ with $\|\sigma\| \le 1$ 
   the mapping
   $$x\mapsto \sigma^{1/2} f(x)  \sigma^{1/2}+ (1-\sigma)^{1/2} g(x) (1-\sigma)^{1/2} $$
   belongs to $\cl F$. This implies  that $\cl F$ is   convex.
   
   Let  $q: C \to C/\cl I$ denote the quotient map and let
   $$q(\cl F)=\{ qf\mid f\in \cl F\}.$$

    Then Arveson's principle 
    (see \cite[p. 351]{[Ar4]}) 
    can be stated like this:

    \begin{thm}[Arveson's principle]
    Assume $E$ separable and $\cl F$ admissible. 
    For the topology of pointwise convergence
   on $E$ we have
   $$\ovl{q(\cl F) } =q(\ovl{ \cl F  } ).$$
   Actually we do not even need to assume $\cl F$ bounded
   if we restrict to the pointwise convergence on a countable
   subset of $E$.
   \end{thm}
   One can verify this assertion by examining 
     the presentations \cite[p. 266]{[Da2]}  or \cite[p. 46 and p. 425]{P4}
     or \cite[Th. 9.46]{P6}.

    The classical admissible classes 
     are contractions,  complete contractions and, when $E$ is an operator system, positive contractions or
     completely positive (c.p. in short)  contractions.
      In the unital case,  unital positive or 
     unital completely positive (u.c.p. in short) maps form admissible classes.
     Let $f,g\in B_{{MB}(E,C)}$. Then it is easy to see on one hand that
     the map $x \mapsto (f(x),g(x))$ is in $B_{{MB}(E,C\oplus C)}$.
     On the other hand the map $\psi: C\oplus C \to C$ defined
     by $v(a,b)=   \sigma^{1/2} a  \sigma^{1/2}+ (1-\sigma)^{1/2} b (1-\sigma)^{1/2} $ is unital c.p.  and hence  in $B_{{MB}(C\oplus C,C )}$ (see Remark \ref{rw}). This shows (with Remark \ref{rc}) that the unit ball of ${MB}(E,C)$ is admissible.

     In the rest of this section,
     we record a few elementary facts.
     
   \begin{lem}\label{wku} Let $A$ be a unital $C^*$-algebra. Let $E\subset A$ be an operator system.
   Consider a unital map $u: E \to C/\cl I$ ($C/\cl I$ any quotient
   $C^*$-algebra).
   Let $v\in MB( E , C)$ (resp.   $v\in SB( E , C)$, rresp. $v\in CB( E , C)$)
   be a  lifting of $u$. Then for any $\vp>0$ there is a 
   unital lifting $v'\in MB( E , C)$ (resp.   $v'\in SB( E , C)$, rresp. $v'\in CB( E , C)$)
   with $\|v'\|_{mb} \le (1+\vp)\|v\|_{mb}$  (resp.  $\|v'\|_{sb} \le  (1+\vp)\|v\|_{sb}$,
    rresp. $\|v'\|_{cb} \le  (1+\vp)\|v\|_{cb}$).\\
   If $v$ is merely assumed bounded, we also have 
   $\|v'\|  \le (1+\vp)\|v\| $.
   \end{lem}
   \begin{proof}
   Let $(\sigma_\a)$ be a quasicentral approximate unit in $\cl I$ in the sense of \cite{[Ar4]} (see also e.g.
\cite{[Da2]} or \cite[p. 454]{P6}).
Let $T_\a:C \oplus \CC \to C$ be the unital c.p. map
defined by $$T_\a(x, c)=\big(\begin{matrix} (1-\sigma_\a)^{1/2} \sigma_\a^{1/2} \end{matrix}\big) \left(\begin{matrix} x \ \ 0 \\ 0 \ \ c
  \end{matrix}\right) \left(\begin{matrix} (1-\sigma_\a)^{1/2} \\\sigma_\a^{1/2}\end{matrix}\right)
=(1-\sigma_\a)^{1/2} x(1-\sigma_\a)^{1/2} + c\sigma_\a.$$
Since $v(1)-1\in \cl I$ we have 
\begin{equation}  \label{27/8b}
 \|(1-\sigma_\a)^{1/2} [v(1)-1](1-\sigma_\a)^{1/2}\|\to 0.\end{equation}
Let $f: E \to \CC$ be a state, i.e. a positive linear form such that $f(1)=1$.
Note that   $\|f\|_{sb}= \|f\|_{mb}=\|f\|=1$. Also $\|v\|\ge \|u\|\ge 1$.
The map $\psi:  E \to C \oplus \CC$ defined by
$\psi(x)= (v(x) , f(x))$ clearly  satisfies $\|\psi\|_{mb}\le \|v\|_{mb}$.
Let $T_\a'(x)= T_\a\psi(x)$. 
Then  $\|T_\a'\|_{mb}\le \| v\|_{mb}$ (see Remarks \ref{rc} and \ref{rw}),  $T_\a': E \to C$ lifts $u $
  and  by \eqref{27/8b} we have
$\|T_\a'(1)-1\|\to 0$, so that going far enough in the net
we can ensure that $T_\a'(1)$ is   invertible in $C$.
We then define $v': E \to C$ by 
$v'(x)= {T_\a'(1)}^{-1} T_\a'(x)  $ (we could use $x\mapsto {T_\a'(1)}^{-1/2} T_\a'(x){T_\a'(1)}^{-1/2} $).
Now $v'$ is unital  and 
$\| v'\|_{mb} \le \|{T_\a'(1)}^{-1}\|  \|T_\a'\|_{mb}\le \|{T_\a'(1)}^{-1}\|  \| v\|_{mb}  $.
Choosing $\a$ ``large" enough so that 
$\|{T_\a'(1)}^{-1}\|  <1+\vp$, we obtain the desired bound
for the mb-case. The other cases are identical.
  \end{proof}
 
  \def\b{\beta}

 Let $I$ be a directed  set.
 Assume $x_i\in B(H)$ for all $i\in I$ and $x\in B(H)$. 
 Recall that (by definition) $x_i$ tends to $x$ for the strong* operator topology (in short 
 $x=\text{sot*}-\lim  x_i$) if $x_i h \to xh$ and $x^*_i h \to x^*h$ for any $h\in H$.
 
Let $M\subset B(H)$
be a von Neumann subalgebra of $B(H)$.
Assume that there is a 
$C^*$-subalgebra $C\subset M$
such that $C''=M$.
It is well known that the unit ball of $M$
is the closure of that of $C$ for the strong* topology (see \cite[Th. 4.8. p. 82]{Tak}).
This implies 
  the following well known and elementary fact representing  $M$ as a quotient
of a natural 
 $C^*$-subalgebra of $\ell_\infty(\{C\})$. 
 
\begin{lem}\label{11/10} In the preceding situation, for a suitable directed index set $I$,    
with which we set $C_i=C$ for all $i\in I$, 
    there is a $C^*$-subalgebra $L \subset \ell_\infty(\{C_i\mid i\in I\})$
     and a surjective     $*$-homomorphism $Q: L
    \to M$ such that for any $(x_i)\in L$ we have
    $Q( (x_i))=  \text{sot*}-\lim  x_i$. \\
    If $C$ is unital, we can get $L$ and $Q$ unital as well.
\end{lem}
\begin{proof} Choose a directed set $I$ so that
for any $x\in B_M$ there is $(x_i)_{i\in I} \in \prod_{i\in I} B_{C_i}$ (recall  $C_i=C$ for all ${i\in I}$)
such that $x=\text{sot*}-\lim  x_i$. For instance, the set of neighborhoods
of $0$ for the {sot*}-topology can play this role. Let $L\subset \ell_\infty(\{C_i\})$
denote the unital $C^*$-subalgebra formed of the elements $(x_i)_{i\in I}$ in $L$ such that
$\text{sot*}-\lim  x_i$ exists. We then set
  $Q(x)=  \text{sot*}-\lim  x_i$ for any $x=(x_i)\in L$. The last assertion is obvious.
\end{proof}

 Let $X\subset A$ be a separable subspace of a $C^*$-algebra $A$.
  For any $C^*$-algebra $C$ we give ourselves
   an admissible class of mappings $\cl F(X,C)$ that we always assume
   closed for pointwise convergence.
   We will say that $X\subset A$ has the  $\cl F$-lifting property
   if for any $u\in \cl F( X , C/\cl I)$   there is a lifting
   $\hat u$ in $\cl F( X ,C)$.
   
  \begin{rem}\label{2/12}  
  In    the sequel, we will assume 
  that $\cl F$ is admissible and  in addition that for any $u\in \cl F( X , C)$
  and any $*$-homomorphism $\pi: C \to D$ between $C^*$-algebras
 $\pi\circ u \in \cl F(X,D)$. When $\cl F$ is formed of unital maps
 ($X$ is then an operator system and $C$ a unital $C^*$-algebra),
 we assume moreover that $\pi\circ u \in \cl F(X,D)$ for any u.c.p.
 map $\pi: C \to D$.
 \end{rem}
 \begin{rem}\label{19/10}[Examples] Examples of such classes are those formed of maps that are
 contractive, positive and contractive, $n$-contractive, $n$-positive and $n$-contractive,
 completely contractive or  c.c.p.  
   When dealing with positive or c.p. maps,   we  
 assume that $X$ is an operator system.
 Then
 we may intersect the latter classes with that of unital ones.
\end{rem}

To tackle liftings within biduals we will use the following classical fact (for which a proof can be found
  e.g.  in \cite[p.465]{P6}): for any 
$C^*$-algebra $C$ and any ideal $\cl I\subset C$ with which we can form the
quotient $C^*$-algebra    $C/\cl I$,  
we have a canonical isomorphism:

\begin{equation}  \label{11/9/4}
C^{**} \simeq (C/\cl I)^{**} \oplus  {\cl I}^{**}\end{equation}
 and hence an isomorphism
$$(C/\cl I)^{**}\simeq  C^{**}/{\cl I}^{**}. $$

The next two statements (which follow easily from Arveson's principle) show that the global (resp. local) lifting property is equivalent
to some sort of  ``global  (resp. local) reflexivity principle".

 \begin{thm}\label{9/10}    Let $X\subset A$ and $\cl F$  be as in Remark \ref{2/12}. Then the following are equivalent:
    \begin{itemize}
 \item[{\rm (i)}] The space $X $ has the $\cl F$-LP.
      \item[{\rm (ii)}] For any von Neumann algebra $M$, any $C\subset M$
    such that $C''=M$ and any $u\in \cl F( X , M)$
    there is a net of maps $u_i \in \cl F( X , C)$ tending pointwise sot*
  to $u$.
   \item[{\rm (iii)}]  For any $C$ and any  $u\in \cl F( X , C^{**})$
  there is a net of maps $u_i \in \cl F( X , C)$ tending pointwise  to $u$
   for the weak*  (i.e. $\sigma(C^{**},C^{*})$) topology.

\end{itemize}
    \end{thm}
    \begin{proof}
     Assume (i). 
    We apply Lemma \ref{11/10}. 
    We have a $C^{*}$-subalgebra $L\subset \ell_\infty(\{C_i\})$ and a surjective  $*$-homomorphism $Q: L
    \to M$ such that for any $(x_i)\in L$ we have $Q( (x_i))=  \text{sot*}-\lim  x_i$ (and a fortiori
    $Q( (x_i))=  \text{weak*}-\lim  x_i$). Let $u\in \cl F( X , M)$ and let
    $\hat u\in \cl F( X , L  )$ be a lifting for $Q$.
    Then $(u_i)_{i\in I}$ such that $\hat u=(u_i)_{i\in I}$ gives us (ii).\\ (ii) $\Rightarrow$ (iii) is obvious. \\
    Assume (iii).  Let $u\in \cl F( X , C/\cl I)$.
    We will use \eqref{11/9/4}.
 The proof can be read on the following diagram. 
 $$\xymatrix{&  & C\ar[d]^{q}  \ar[r] & C^{**}\ar[d]^{q^{**}} \\
   &  X 
   \ar@/^3pc/[urr]^{v} \ar@{-->}[ur]_{v_i}
  \ar[r]_{ u \quad  }  &  C/\I  \ar[r] & [C/\cl I]^{**} \ar@/_2pc/[u] _{\rho}}$$
 By \eqref{11/9/4} we have a lifting  $\rho: (C/\cl I)^{**}\simeq  C^{**}/{\cl I}^{**} \to C^{**}$
which is a $*$-homomorphism.
Let $f$ be a state on $(C/\cl I)^{**}$ and let $p$ denote the unit in ${\cl I}^{**}\subset C^{**}$.
 Then $\rho_1: C^{**}/{\cl I}^{**} \to C^{**}$
defined by $\rho_1 =\rho(\cdot) + f(\cdot) p$ is a u.c.p. lifting
(replacing $\rho$ by $\rho_1$ is needed only when $\cl F$ is formed of unital maps).
 Therefore, recalling Remark \ref{2/12},  we can    find a map $v\in {\cl F}( X , C^{**})$  
such that  $q^{**} v=i_{C/\cl I} u$ (where as usual $i_D: D \to D^{**}$ the canonical inclusion). 
By (iii)  there is a net
$(v_i)$   in   $ {\cl F}( X , C )$  tending weak* to $v$.
Then $qv_i=q^{**} v_i $ tends  pointwise {weak*} to $q^{**}v=u$.
This means $qv_i \to u$ pointwise for the weak topology of $C/\cl I$.
We can then invoke   Mazur's classical theorem 
 to obtain (after passing to suitable convex combinations) a net such that
$qv_i \to u$ pointwise in norm (see e.g. \cite[Rem. A.10]{P6} for details).
By Arveson's principle,  $u$ admits a lifting in $\cl F$, so we obtain (i).
\end{proof}

Let $X\subset A$ be as above. In addition
for any $C$ and any f.d. subspace $E_0\subset X$ we assume given 
another f.d.  $E\supset E_0$ and a class
$\cl F(E,C)$ satisfying the same two assumptions as $\cl F(X,C)$.  
We also assume that $u\in \cl F(X,C)$ implies $u_{|E}\in \cl F(E,C)$.
 We will say that $u: X \to C/\cl I$ is locally $\cl F$-liftable
 if for any f.d.   $E_0\subset X$
there is  a f.d.  $E\supset E_0$ and a map $u^E\in \cl F(E,C)$ lifting the restriction $u_{|E}$.
 We will say that $X$ has the $\cl F$-LLP if
for any $C$, any $u\in \cl F(X,C/\cl I)$ is  locally $\cl F$-liftable. The introduction of $E$ in place of $E_0$ is a convenient way
to include e.g. the class $\cl F$ formed of unital c.p. maps on f.d. operator systems.
In the latter case $E$ can be any f.d. operator system containing $E_0$.

   \begin{thm} With the above assumptions and those of Theorem \ref{9/10},  the following are equivalent:
    \begin{itemize}
 \item[{\rm (i)}] The space $X $ has the $\cl F$-LLP.
  \item[{\rm (ii)}]  For any $C$,  any  $u\in \cl F( X , C^{**})$ and any f.d.    $E_0\subset X$, there is  a f.d.  $E\supset E_0$ and
  a net of maps $u^E_i \in \cl F( E, C)$ tending pointwise weak*
  to $u_{|E}$.
   \end{itemize}
    \end{thm}
    \begin{proof} The proof of Theorem \ref{9/10} can be easily adapted to prove this.
\end{proof}
     
Since the work of T.B. Andersen and Ando \cite{TBA,An} it has been known that 
if a separable Banach space $X$ has 
the metric approximation property
then
any contractive $u: X \to C/\cl I$ admits a contractive lifting.
Although we will not use it,
we  conclude our general discussion
by reformulating this result in our framework.
Let us say that $u: X \to C/\cl I$ is $\cl F$-approximable
if there is a net of finite rank maps $u_\a\in \cl F (X, C/\cl I)$
tending pointwise to $u$. 
For the examples of  $\cl F$ considered in Remark \ref{19/10}
it is obviously equivalent to say that  the identity map of $C/\cl I$  is locally $\cl F$-liftable
or to say that  for any $X$ any finite rank map  in $\cl F( X , C/\cl I)$ is locally $\cl F$-liftable.

With this terminology, Ando's 
well known result   \cite{An}
can nowadays be reformulated like this:
\begin{pro} Assume that any finite rank map in $  \cl F( X , C/\cl I)$ is liftable in $\cl F(X,C)$.
If $u: X \to C/\cl I$ is $\cl F$-approximable, 
  then
$u$ admits a (global) lifting in $\cl F (X, C )$.
\end{pro}
\begin{proof} This follows from Arveson's principle.
 \end{proof}

 Let us denote by $\cl F_{ccp}$ (resp. $\cl F_{ucp}$) the class of
 c.c.p. (resp. u.c.p.) maps.
 Then the LP is just the $\cl F_{ccp}$-LP.
 In the unital case, we   show below that the $\cl F_{ccp}$-LP
 and the  $\cl F_{ucp}$-LP are equivalent.
 Of course in the latter case we restrict to
 lifting quotients of unital $C^*$-algebras.

  The following statement about the unitization process is a well known
  consequence of the works of Choi-Effros 
  and Kirchberg \cite{[CE5], Kiuhf}.
  
  \begin{pro}\label{14/11} Let $A$ be   a separable $C^*$-algebra and let $A_1$ be its unitization. The following properties of    $A$
  are equivalent:
  \item{\rm (i)} The lifting property LP, meaning the $\cl F_{ccp}$-LP.
   \item{\rm (ii)}  The unitization $A_1$ has the $\cl F_{ucp}$-LP.
\\Moreover, when $A$ is unital {\rm (i)} and {\rm (ii)} are equivalent to
     \item{\rm (iii)} The $C^*$-algebra $A$ has the $\cl F_{ucp}$-LP.
     \end{pro}
  \begin{proof} 
Assume (i). Let $q: C \to C/\cl I$ be the quotient map.
  Assuming $C$ unital, let $u: A_1\to C/\cl I$ be a u.c.p. map.
  By (i) there is $w \in CP(A , C)$ with $\|w\|\le 1$
  such that $qw=u_{|A}: A \to C/\cl I$. Let $w_1: A_1 \to C_1$ be the unital extension of $w$.
  By \cite[Lemma 3.9]{[CE5]}, $w_1$ is c.p. Since $C$ is unital
  there is a unital $*$-homomorphism $\pi: C_1 \to C$ extending  $Id_C$.
  Let $\hat u= \pi w_1: A_1 \to C$. Then $\hat u$ is a u.c.p. map
  such that ${q \hat u}_{|A}=u_{|A}$   and $q \hat u(1)=1$. It follows that
  ${q \hat u}=u$, whence (ii).
  
    Assume (ii). Let $u: A \to C/\cl I$ be a c.c.p. map.
    Let $u_1: A_1 \to (C/\cl I)_1\simeq C_1/\cl I$ be the unital map extending $u$.
    By \cite[Lemma 3.9]{[CE5]} again, $u_1$ is c.p. By (ii)
    there is a unital c.p. map $\widehat {u_1}: A_1 \to C_1$ lifting $u_1$.
    Let $q: C \to C/\cl I$  and $Q: C_1 \to C_1/\cl I\simeq (C/\cl I)_1$ be the quotient maps. 
   We have $Q \widehat {u_1}=u_1$ and ${u_1}_{|A} =u$.
   A moment of thought shows that
   $Q$ is the unital extension of $q$, so that
   $Q^{-1} ( C/\cl I)=C$.
   Thus $Q \widehat {u_1}(a) =u_1(a)=u(a) \in C/\cl I$ for any $a\in A$
   implies that $\widehat {u_1}(a) \in C$  for any $a\in A$.
We conclude that ${\widehat {u_1}}_{|A} : A \to C$ is a  
  c.p. lifting of $u$ with norm $\le 1$, whence (i).
  
    Assume (i) with $A$ unital. Let $u: A \to C/\cl I$ be a u.c.p. map.
    Let $v:A \to C$ be a c.c.p. lifting of $u$.
    To show (iii) 
     let $f$ be a state on $A$ and let
    $\hat u= v + (1-v(1)) f$. Then $\hat u$ is a u.c.p. lifting.
    This shows (i) $\Rightarrow$ (iii).

     Conversely, assume (iii). To show (i) let $u: A \to C/\cl I$ be a c.c.p. map.
     Let $u': A \to C_1/\cl I$ be the map $u$ composed with the inclusion
     of $C$ into its unitization $C_1/\cl I$.
     Note that since $u'$ takes its values in $C/\cl I$,
     any lifting $\hat u$ of $u'$ must take its values in $C$, and hence be a lifting of $u$.
     Therefore, it suffices to show that $u'$ admits a c.c.p.  lifting.
      Assume first that $u'(1)$ is invertible in $C_1$.
      We will argue as for Lemma \ref{wku}.
     Define  $w\in CP( A , C_1/\cl I)$ by $w(x)= u'(1)^{-1/2} u'(x)u'(1)^{-1/2}$.
     Then $w$ is unital. Let $\hat w: A \to C_1$ be a u.c.p. lifting of $w$.
     Let $z\in C_1$ be a lifting of $u'(1)$ with $\|z\|=\|  u'(1)\| \le 1$.
     Then the mapping $\hat u: A \to C_1$
     defined by $\hat u(\cdot) = z^{1/2} \hat w (\cdot)  z^{1/2}$ is a
     c.c.p.  lifting of $u'$.  
     Let $f$ be a state on $A$. To complete the proof
   note that for any $\vp>0$ the map $u_\vp\in CP(A,C_1/\cl I)$
   defined by $u_\vp(\cdot)=u'(\cdot) +\vp f(\cdot)$ is a c.p. 
   perturbation of $u'$ with $u_\vp(1)$ invertible.
   By what precedes the maps $(1+\vp)^{-1} u_\vp $ admit
    c.c.p.  liftings. By Arveson's principle
    $u'$ also does. This shows (iii) $\Rightarrow$ (i).
\end{proof}
 \begin{rem} If $A$ is unital (i)-(iii) in Proposition \ref{14/11} are also equivalent to:
 \item{\rm (iv)}  Any unital $*$-homomorphism $u: A \to \C/\cl I$ into a quotient of $\C$
      admits a u.c.p. lifting.\\
      Indeed, (iii) $\Rightarrow$ (iv) is trivial, and  to show (iv) $\Rightarrow$ (iii) one
      can realize $A$ as a quotient of $\C$.
  Then (iv) implies that the identity of $A$ factors via 
  u.c.p. maps through $\C$, so that   (iii) follows from  
  Kirchberg's theorem that    $A=\C$ satisfies (iii) or (ii). \\
 However we  deliberately avoid using the equivalence with (iv)
 to justify our claim that Theorem \ref{L1} yields a new proof
 of the latter  theorem of Kirchberg.
 \end{rem}
 \begin{rem} In \cite {Kir}, Kirchberg defines the LP for $A$  by the property
 (ii) in Proposition \ref{14/11}.  We prefer to use
the  equivalent definition in (i) that avoids   the unitization.
 \end{rem}
 See \cite{HWW}
 for  a discussion of lifting properties in the more general context of $M$-ideals.

\section{A ``new" extension theorem}

We start by a new version of the ``local reflexivity principle" (see particularly  \eqref{gty}). This is the analogue of \cite[Lemma 5.2]{EH}
for the maximal tensor product.

 \begin{thm}\label{t3} Assume that $A$ satisfies  \eqref{pty}.
Let $E\subset A$ be any f.d. subspace.
Then  for any $C^*$-algebra $C$ we have a contractive inclusion
\begin{equation}\label{mblr}
{MB}(E,C^{**})\to  {MB}(E,C)^{**}.\end{equation}
In other words any $u$ in the unit ball of ${MB}(E,C^{**})$ is the weak* limit
of a net $(u_i)$ in the unit ball of ${MB}(E,C)$.
\end{thm}
   \begin{proof}
  This will follow from the bipolar theorem.
  We first need to identify the dual of
   ${MB}(E,C)$. As a vector space 
   ${MB}(E,C) \simeq C  \otimes E^*$ and
   hence ${MB}(E,C)^* \simeq C^* \otimes E$
(or say $({C^*})^{\dim(E)}$).
  We equip $C^* \otimes E$ with the norm $\a$ defined as follows.
  Let $K\subset {MB}(E,C)^*$ denote the set of those $f \in {MB}(E,C)^*$ for which 
  there is a $C^*$-algebra $D$,
   a functional  $w$ in the unit ball of
   $ {(D\otimes_{\max } C)^* } $
  and  
  $t \in B_{ D\otimes_{\max } E}$, 
  so that
  $$\forall u\in {MB}(E,C)\quad f(u)= \langle w, [Id_D \otimes u](t)\rangle .$$
   We could rephrase this in tensor product language:
  note that we have a natural bilinear map
$$ {(D\otimes_{\max } C)^* } \times (D \otimes_{\max} E) \to C^* \otimes E$$
associated to the duality $D^* \times D \to \CC$, then $K$
can be identified with  the union (over all $D$'s) of the  images of the product of the two unit balls
under this bilinear map.

We will show that $K$  is  the unit ball of
${MB}(E,C)^*$.
Let $D_1,D_2$ be $C^*$-algebras. Let $D=D_1\oplus D_2$
with the usual $C^*$-norm.
Using the easily checked identities
 (here the direct sum is in the $\ell_\infty$-sense) $$D \otimes_{\max} E= (D_1 \otimes_{\max} E) \oplus (D_2 \otimes_{\max} E), \text{  and  } {D\otimes_{\max } C  }=(D_1\otimes_{\max } C ) \oplus
(D_2\otimes_{\max } C) ,$$
and hence $ {(D\otimes_{\max } C)^* }=(D_1\otimes_{\max } C)^* \oplus_1
(D_2\otimes_{\max } C)^*$
(direct sum in the $\ell_1$-sense), it is easy to check that 
$K$ is convex and hence that $K$
is the unit ball of some norm $\a$ on   ${MB}(E,C)^*$.\\
  Our main point  is the  claim that $K$ is weak* closed.
  To prove this, let $(f_i)$ be a net in $K$ converging
  weak* to some $f \in {MB}(E,C)^*$.
  Let $D_i$ be $C^*$-algebras, $w_i \in B_{(D_i\otimes_{\max} C)^*}$
  and $t_i \in B_{D_i \otimes_{\max} E}$ such that    we have 
  $$\forall u\in {MB}(E,C)\quad f_i(u)= \langle w_i, [Id_{D_i} \otimes u](t_i)\rangle .$$
    Let $D= \ell_\infty(\{D_i\})$ and 
  let $t\in D\otimes E$ be associated to $(t_i)$.
   By \eqref{pty'}
  we know that $\|t\|_{\max}\le 1$.
  Let $p_i: D \to D_i$ denote the canonical coordinate projection,
  and let $ v_i\in (D\otimes_{\max} C)^*$ be the functional
  defined by $v_i(x)= w_i(  [p_i\otimes Id_C](x) )$.
  Clearly $v_i \in B_{(D\otimes_{\max} C)^*}$  and
 $  f_i(u)= \langle v_i, [Id_D \otimes u](t)\rangle .$
 Let $w$ be the weak* limit of $(v_i)$.
  By   weak* compactness,  $w\in B_{(D\otimes_{\max} C)^*}$.
  Then
  $f(u)=\lim f_i(u)= \langle w, [Id_D \otimes u](t)\rangle .$
  Thus we conclude $f\in K$, which proves our claim.
  
  By the very definition of $\|u\|_{{MB}(E,C)}$ we have
  $$\|u\|_{mb}= \sup\{ |f(u) | \mid f\in K\}.$$
  This implies that the   unit ball of the dual of ${MB}(E,C)$ is the bipolar of $K$,
  which is equal to its weak* closure. By what precedes,
  the latter coincides with $K$. Thus the gauge of $K$ is the 
  announced dual norm $\a=\|\  \|_{{MB}^*}$.
  
  Let $u''\in {{MB}(E,C^{**})}$ with $\|u''\|_{mb}\le 1$. By the bipolar theorem,
  to complete the proof  it suffices to show that
  $u''$ belongs to the bipolar of $K$, or equivalently that
  $ |f(u'')|\le 1  $ for any $f\in K$.
  To show this consider  $f\in K$ taking 
  $  u\in MB(E,C)$ to $ f(u)= \langle w, [Id_D \otimes u](t)\rangle $
  with $w \in B_{(D\otimes_{\max} C)^*}$
  and $t \in B_{D \otimes_{\max} E}$.
  Observe that $[Id_D \otimes u''] (t) \in D \otimes C^{**} \subset (D \otimes_{\max} C)^{**}$.
  Recall that $MB(E,C^{**})\simeq MB(E,C)^{**} \simeq (C^{**})^{\dim(E)}$
  as vector spaces. Thus we may view  $f\in MB(E,C)^{*}$ 
  as a weak* continuous functional on $MB(E,C^{**})$ to define 
$f(u'')$.
   We claim that 
  \begin{equation}\label{cb1}  f(u'')= \langle w, [Id_D \otimes u''](t)\rangle ,\end{equation}
  where the duality is relative  to
  the pair $\langle (D \otimes_{\max} C)^{*},(D \otimes_{\max} C)^{**} \rangle$.
  From this claim the conclusion is immediate. Indeed,
  we have 
 $\|[Id_D \otimes u''](t) \|_{D \otimes_{\max} C^{**}}\le \|u''\|_{mb}\le 1$, and
by the maximality of the max-norm on $D  \otimes  C^{**} $
we have a fortiori 
$$\|[Id_D \otimes u''](t) \|_{(D \otimes_{\max} C)^{**} }\le \|[Id_D \otimes u''](t) \|_{D \otimes_{\max} C^{**}} \le 1.$$
Therefore  
$|f(u'')|=|\langle w, [Id_D \otimes u''](t)\rangle|  \le \|w\|_{(D \otimes_{\max} C)^{*}} \le 1$, which completes the proof modulo our claim \eqref{cb1}.\\
To prove the claim, note that 
the identity \eqref{cb1}  holds for any $u\in MB(E,C)$.
Thus it suffices to prove that the right-hand side of \eqref{cb1} 
is a  weak* continuous function of $u''$ (which is obvious for the left-hand side). 
To check this
one way is to note that $t\in D \otimes E$ can be written
as a finite sum $t=\sum d_k \otimes e_k$ ($d_k\in D,e_k\in E$) and 
if we denote by $\dot w: D \to C^*$ the linear map associated to $w$
we have
$$\langle w, [Id_D \otimes u''](t)\rangle
=\sum\nl_k \langle w, [d_k \otimes u''(e_k)]\rangle
=\sum\nl_k  \langle  \dot w(d_k) , u''(e_k) \rangle, $$
and since  $\dot w(d_k) \in C^*$ the weak* continuity as a function of $u''$ is obvious,
completing the proof.
  \end{proof}

\begin{rem}\label{fg} The preceding proof actually shows that, without any assumption on $A$ or $C$,
 we have a contractive inclusion
  \begin{equation}\label{gty}{SB}(E,C^{**})\to {SB}(E,C)^{**}\end{equation}
 Of course if $A$ satisfies \eqref{pty} then ${SB}(E,C)= {MB}(E,C)$ isometrically and we  recover Theorem \ref{t3}.
\end{rem}

\begin{rem}  
The converse inclusion to \eqref{mblr} holds in general : we claim that
 we have a contractive inclusion \begin{equation}\label{gtyb}{MB}(E,C)^{**}\to {MB}(E,C^{**}).\end{equation}
 Indeed, let $u_i: E \to C$ be  a net with $\|u_i\|_{mb} \le 1$
 tending weak* to $u: E \to C^{**}$.  Let $t\in \C \otimes E$ with $\|t\|_{\max} \le 1$,
 say with $t\in F \otimes E$ with $F\subset \C$ f.d. Then  $[Id_{\C} \otimes u] (t) \in B_{[F \otimes_{\max} C]^{**}}$.
 Note that since $F$ is f.d. $[F \otimes_{\max} C]^{**} \subset [\C \otimes_{\max} C]^{**} \cap [F \otimes  C^{**}] $.
We have clearly an isometric inclusion $[F \otimes_{\max} C]^{**} \subset [\C \otimes_{\max} C]^{**}$.
Moreover (see \cite[Th. 8.22]{P6}) 
we have an  isometric
embedding $ \C^{**} \otimes_{\rm bin} C^{**}
\subset [\C \otimes_{\max} C]^{**}$.
Since $F\subset \C$, the norm induced by $\C^{**} \otimes_{\rm bin} C^{**}$
on $F \otimes  C^{**}$ coincides with that induced by $\C  \otimes_{\rm nor} C^{**}$.
By Kirchberg's theorem  to the effect that \eqref{27/9} holds when $A=\C$ (see \cite[Th. 9.10]{P6})
  $\C  \otimes_{\rm nor} C^{**}= \C  \otimes_{\max} C^{**}$.
  Thus we find $\|u\|_{MB(E, C^{**})} \le 1$.
  This proves the claim.
\end{rem}

\begin{rem}\label{per} We will use  an elementary
perturbation argument as follows. Let $A,C$ be  $C^*$-algebras and let $E\subset A$
be a f.d. subspace. Let $v: E \to C$. For any $\d>0$ there is $\vp>0$
(possibly depending on $E$)
satisfying the following: for any
 map  $v': A \to C$ be such that $\|v'_{|E}-v\| \le  \vp$,
  there is a map $v'': A \to C$ such that $$v''_{|E}= v \text{  and  }
 \|v''-v'\|_{ mb} \le \d , 
  \text{  and hence  } \|v''\|_{mb} \le   \|v'\|_{mb} +\d. $$
 $$\xymatrix{&A\ar@{-->}[dr]^{ v'} \\
& E\ \ \ar@{^{(}->}[u] 
 \ar[r]_{\ \  {v'}_{|E}\approx v \quad } & C} \xymatrix{&A\ar@{-->}[dr]^{ v''\approx v'} \\
& E\ \ \ar@{^{(}->}[u] 
 \ar[r]_{ v \quad } & C}$$
  Let $\Delta= v-v'_{|E}$. 
  Let $\|\Delta\|_{N} $
  denote
   the  nuclear norm (in the Banach space sense) 
  of $\Delta:E \to C$. By definition (here $E$ and $C$ are Banach spaces with  $E$ f.d.), this is the infimum
  of $\sum_1^d \|f_j\|_{E^*} \|c_j\|_{C}$ over all 
  the possible representations of $\Delta$ as $\Delta(x)=\sum_1^d f_j(x) c_j$ ($x\in E$).
  Let $k_E=\|Id_E\|_{N}$. 
  It is immediate that for any $\Delta :E \to C $ we have
 $\|\Delta\|_{N} \le k_E \|\Delta\|$.
 By Hahn-Banach, $\Delta$ admits an extension $\tilde\Delta: A \to E$
  with $\|\tilde\Delta\|_{N} \le \| \Delta\|_{N} \le k_E \|\Delta\| \le \vp k_E$.
  Let $v''= v' + \tilde \Delta$. Then $v''_{|E}= v$
  and $$ \|v''-v'\|_{mb} \le \|v''-v'\|_{N}=\|\tilde\Delta\|_{N} \le \vp k_E.$$
  Whence the announced result with $\d=\vp k_E.$
\end{rem}

\begin{thm}\label{ex} Assume that $A$ satisfies  \eqref{pty} (or merely the conclusion of Theorem \ref{t3}).
Let   $C$ be a $C^*$-algebra. Let $E\subset X$ be a f.d. subspace of a separable subspace $X\subset A$.
Then for any $\vp>0$,  any map $v: E \to C$  admits an extension
$\tilde v: X \to C$ such that
$\|\tilde v \|_{mb} \le (1+\vp)\|  v \|_{mb} $.
\end{thm}

$$\xymatrix{&X\ar@{-->}[dr]^{ \tilde v } \\
& E\ \ \ar@{^{(}->}[u] 
 \ar[r]^{ v \quad } &{  \ \ C}  }$$
\begin{proof}
Let $E_1\supset E$ be any finite dimensional superspace with $E_1\subset X$.
  Our first goal will be to show that for any $\vp>0$
  there is an extension $w: E_1 \to C$ such that $w_{|E}=v$
  with $\|w\|_{mb} \le  \|v\|_{mb}+\vp$.
\\
  By Theorem \ref{ext} and by \eqref{e22},  there is a  decomposable map   $\tilde v: A \to C^{**}$
  such that $$\|\tilde v\|_{D(A , C^{**})} =\|\tilde v\|_{MB(A , C^{**})}\le \|v\|_{mb}$$ which is an
   extension of $v$ in the sense that $\tilde v_{|E}=i_C v$
   where $i_C: C \to C ^{**}$ is the canonical inclusion.
   Let $v_1: E_1 \to C^{**}$ be the restriction of $\tilde v$ so  that
   $v_1=\tilde v_{|E_1}$.
   A fortiori  $\|v_1: E_1 \to C^{**}\|_{mb} \le \|v\|_{mb}$.
 By Theorem \ref{t3} there is a net of maps
 $v^i : E_1 \to C   $    with $\| v^i \|_{mb} \le \|v\|_{mb}$ that tend
 weak* to $v_1$. It follows that the restrictions $v^i_{|E}$ tend
 weak* to ${v_1}_{|E}= \tilde v_{|E}=i_C v$.
 This means that
 $v^i_{|E}(x)$ tends weakly in $C$ (i.e.  for $\sigma(C, C^*)$) to $v(x)$  for any $x\in E$ .
    By a well known application of Mazur's theorem
    (see e.g. \cite[Rem. A.10]{P6} for details),
    passing to suitable convex combinations of the $v^i$'s, 
    we can get a similar net such that in addition $v^i_{|E} $ tends 
 pointwise in norm to $v$. Since $E$ is f.d. this implies
 $\| v^i_{|E} -v\|\to 0$.
 By  perturbation (see Remark  \ref{per}), for any $\vp >0$
 we can find  $w: E_1 \to C$ such that $w_{|E}=v$
  with $\|w\|_{mb} \le \|v\|_{mb}+\vp$, so we reach our first goal.\\
Lastly we use the separability of $X$ to form an increasing sequence
$(E_n)$ of f.d. subspaces of $X$ with dense union
such that $E_0=E$. Let $w_0=v : E_0 \to C$.
By induction, we can find a sequence of maps $w_n: E_n \to C$
such that ${w_{n+1}}_{|E_n} = w_n $
and $\|w_{n+1}\|_{mb} \le \|w_{n}\|_{mb} + 2^{-n-1} \d$.
By density this extends to a linear operator
  $w: X \to C$ such that 
  $\|w\|_{mb} \le \|w_{0}\|_{mb} + \d =\|v\|_{mb} + \d $. 
  This   completes the proof.
\end{proof}

\begin{rem} Actually,   it suffices   by Remark \ref{fg} to assume that $X\subset A$ 
satisfies \eqref{pty} (in place of $A$) for the preceding proof to be valid.
We rephrase this in the next statement.
\end{rem}

\begin{thm}\label{t12}
Let $E\subset X$ be a f.d. subspace of a separable subspace $X\subset A$
of a $C^*$-algebra.
For any $\vp>0$,  any map $v: E \to C$ into a $C^*$-algebra $C$ admits an extension
$\tilde v: X \to C$ such that
$\|\tilde v \|_{{sb}} \le (1+\vp)\|  v \|_{{sb}} $.
\end{thm}

\begin{rem}\label{r28/11}
Let $A,C$ be unital $C^*$-algebras. Let $E\subset A$ be an operator system
and let $u: E \to C$ be a unital map.
Recall    (see Remark \ref{rw1}) that $u$ is m.p.
if and only if $\|u\|_{mb}=1$.\\
We will denote by $\cl F_{ump}(E,C)$ the class
of u.m.p. maps $u:E \to C$. This is clearly an admissible class
and it is pointwise closed.
\end{rem}

 We will need 
the following consequence of Theorem \ref{t3} for u.m.p. maps:

\begin{cor}\label{c28/11} Let $A,C$ be   unital $C^*$-algebras.
Assume that $A$
satisfies the conclusion of 
Theorem \ref{t3}. Let $E\subset A$ be a f.d. operator system.
Let $u: E \to C^{**}$ be a
  u.m.p. map.
  There is a net of u.m.p. maps $u_i: E \to C $
   tending pointwise
weak*  to $u$.
\end{cor}
\begin{proof} 
By Theorem \ref{t3}
 there is a net of maps $v_i: E \to C$ with $\|v_i\|_{mb} \le 1$
 tending pointwise  weak* to $u $. 
 Replacing $v_i$ by 
 $ (v_i+(v_i)_*)/2$,
 we may  assume 
 each $v_i$ self-adjoint (see Remark \ref{rw2}).
 Moreover, since $u(1)=1$,  we may observe that $v_i(1)-1 \to 0 $ 
 in the weak topology (i.e. $\sigma(C,C^*)$) of $C$.
 Passing to convex combinations, we may assume (by Mazur's theorem) that
 $\|v_i(1)-1\| \to 0 $. By Lemma \ref{v26}  
 there is a unital map $u_i \in MP(E,C)$ 
 such that $\| u_i -v_i\|\to 0$. Since
 $v_i \to u $ pointwise weak*, we also have
 $u_i \to u $ pointwise weak*.
\end{proof}

\begin{thm}\label{t28/11} Let $A,C$ be   unital $C^*$-algebras.
Assume that $A$
satisfies the conclusion of 
Theorem \ref{t3}. Let $E\subset X$ be a f.d. operator subsystem
of a separable operator system $X\subset A$.
Let $u: E \to C $ be a
  u.m.p. map. Then for any $\vp>0$
  there is a u.m.p. map
  $v : X \to C $
  such that $\| {v}_{|E} -u\|<\vp$.
\end{thm}
\begin{proof}  We may assume that we have an increasing sequence
  of f.d. operator systems $(E_n)$ with
 $E_0=E$ such that $\ovl{\cup E_n}=X$.
 We give ourselves  $\vp_n>0$ such that $\sum_{n\ge 1}  \vp_n<\vp$.
 The plan is to construct a sequence of u.m.p. maps 
 $v_n :  E_n \to C $ such that
  $v_0=u$
 and $\|{v_{n+1}}_{|E_n}- v_n\|  <\vp_{n+1} $ for all $n\ge 0$.
 Then the map $v$ defined by $v(x)= \lim v_n(x)$ for $x\in \cup E_n$
 extends by density to a u.m.p. map   $v: X \to C$ 
 such that
 $\|v_{|E} -u\| \le \sum_{n\ge 1} \vp_n <\vp$.
 
 The sequence $(v_n)$ will be constructed by induction starting from $v_0=u$.
  Assume that we have constructed $v_0,\cdots,v_n$ ($n\ge 0$)
satisfying the announced properties and let us produce $v_{n+1}$.
Since   $\|v_n\|_{mb} =1$ by Theorem \ref{ext} and \eqref{e22}
there is a (still unital) map  $\tilde v_n : A \to C^{**}$ 
such that ${\tilde {v_n} }_{|E_n}=i_C v_n $ with $\|  \tilde {v_n}\|_{mb} =1$.
By Remark \ref{rw1}, the map  ${\tilde {v_n} }$ is u.m.p. (see also Remark \ref{trick}).
A fortiori ${\tilde {v_n} }_{|E_{n+1}}: E_{n+1}  \to C^{**}$ is u.m.p.
 By Corollary  \ref{c28/11}
there is a net $(u_i)$ of u.m.p. maps from $E_{n+1} $ to $C$
such that $u_i  -{\tilde {v_n} }_{|E_{n+1}}$  tends weak* to $0$. 
Since $E_n \subset E_{n+1}$,
 the restrictions $(u_i  -{\tilde {v_n} })_{|E_n}$ also tend weak* to $0$.
 Thus ${u_i }_{|E_n} -v_n $ tends weak* to $0$ when $i\to \infty$.
 Since ${u_i }_{|E_n} -v_n  $ takes values in $C$, this tends to  $0$
 weakly in $C$.
 Passing to convex combinations of the maps $u_i$
 we may assume that $\lim_i\|{u_i }_{|E_n} -v_n \| = 0$.
Choosing $i$  large enough, we have
 $\|{u_i }_{|E_n} -v_n \| <\vp_{n+1}$, so we may set $v_{n+1}=u_i$.
This completes the induction step.
\end{proof}

\section{A ``new" lifting theorem}

 We start by what is now a special case of Theorem
 \ref{9/10}. \\ Let $C/\cl I$ be a quotient $C^*$-algebra with quotient map $q: C \to C/\cl I$

\begin{lem}\label{L21}  Let $E \subset A$  a   f.d. subspace.  
Then 
any $u: E \to C/\cl I$
admits a lifting $\hat u: E \to C$ such that
$$\|\hat u \|_{{sb}} = \|  u \|_{{sb}} .$$
\end{lem}
\begin{proof} 
 
By (iii) $\Rightarrow$ (i) in  Theorem
 \ref{9/10} 
 it suffices  to show that  the classes $$\cl F(E,C)=\{u:E \to C\mid \|u \|_{{sb}}\le 1\}$$
 satisfy the assertion (iii) in Theorem
 \ref{9/10} with $X=E$. This is precisely what \eqref{gty} says.
 \end{proof}

\begin{rem}\label{r21} In the situation of Lemma \ref{L21}, consider
 $u: E \to C/\cl I$. We claim that 
 if   the 
 map ${MB}(E,C^{**})\to  {MB}(E,C)^{**}$ is contractive
   then
   $u$
admits a lifting $\hat u: E \to C$ such that
$$\|\hat u \|_{{mb}} = \|  u \|_{{mb}} .$$
Indeed, just as in the preceding proof, we may apply (iii) $\Rightarrow$ (i) in 
Theorem
 \ref{9/10} to the class $$\cl F(E,C)=\{u:E \to C\mid \|u \|_{{mb}}\le 1\}.$$
\end{rem}

\begin{thm}\label{lift} 
Let $X\subset A$ be a separable subspace of a $C^*$-algebra $A$.
Then any $u\in {SB}(X, C/\cl I)$ admits a lifting
$\hat u: X \to C$ with $\|\hat u\|_{{SB}(X,C)} =\|  u\|_{{SB}(X,C/\cl I)}.$\\
If $A$ satisfies the property \eqref{pty},
then any $u\in {MB}(X, C/\cl I)$ admits a lifting
$\hat u: X \to C$ with $\|\hat u\|_{{MB}(X,C)} =\|  u\|_{{MB}(X,C/\cl I)}.$
\end{thm}

\begin{proof} Let $u\in SB(X, C/\cl I)$. Let $E\subset X$ be a f.d. subspace.
By Lemma \ref{L21}, there is $u^E: E \to C$ lifting $u_{|E}$
with $\|u^E\|_{{SB}(E,C)}\le \| u_{|E}\|_{{SB}(E,C/\cl I)} \le \| u\|_{{SB}(X,C/\cl I)}.$ Let $\vp >0$.
By Theorem \ref{ex}, the map $u^E$ admits an extension
$\widetilde{u^E}: X \to C$ such that
$\|\widetilde{u^E}\|_{{SB}(X,C)} \le (1+\vp) \|u^E\|_{{SB}(E,C)}\le (1+\vp)\| u\|_{{SB}(X,C/\cl I)}.$
Since for any $x\in X$ we have $x\in F$ for some f.d.  $F\subset X$, we have  $q\widetilde{u^E}(x) =u(x)$ whenever $E\supset F$.
Therefore, the map $u:X \to C/\cl I$ is in the pointwise closure
of the liftable maps  $\{ (1+\vp)^{-1} q\widetilde{u^E} \mid E\subset X\}$.
By normalization, the liftings
$\{ (1+\vp)^{-1} \widetilde{u^E} \mid E\subset X\}$
 are in the unit ball of ${SB}(X,C)$ and the latter is an admissible class.
Thus the conclusion
 of the first part 
 follows from Arveson's principle.
 The second part is obvious since \eqref{pty} implies that
 ${{MB}(E,C)}={{SB}(E,C)}$ isometrically for any $C$.
  \end{proof}

Let $A,C$ be unital $C^*$-algebras. Let $E\subset A$ be an operator system.
Recall that $\cl F_{ump}(E,C)$ denotes
the class of unital maps
 in $MP(E,C)$ 
 (see Remarks \ref{rw1} or \ref{r28/11}).

 \begin{thm}\label{26/11}
 Assume that $A$ satisfies  the conclusion of Theorem \ref{t3}.
  Let $X\subset A$ be a separable operator system.
 Any $u\in  \cl F_{ump}( X , C^{**}) $ is the pointwise weak* limit
 of a net in $\cl F_{ump}( X , C) $.
 \end{thm}
 \begin{proof} 
  Let $E\subset X$ be a f.d. operator system.
  By  Corollary \ref{c28/11} there is a net $u_i\in \cl F_{ump}( E , C) $ such that
  $u_i \to u_{|E}$ pointwise weak*.
  By Theorem \ref{t28/11}
  there are maps $v_i\in \cl F_{ump}( X , C) $ such that 
  $\|{v_i}_{|E} -u_i\|\to 0$, and hence $(v_i-u)_{|E} \to 0$ pointwise weak*.
  This implies that $u$ is in the pointwise weak* closure
  of $\cl F_{ump}( X , C) $.
 \end{proof}
 
  By Theorem \ref{9/10} we may state:
  \begin{cor}  If $A$ satisfies \eqref{pty}
  any separable operator system $X\subset A$
  has the $\cl F_{ump}$-LP.
 \end{cor}
 
 \begin{cor}\label{17/11}  Assume $A,C$   unital and $A$   separable.  If  $A$ 
 satisfies \eqref{pty}
 or merely the conclusion of Theorem \ref{t3}, any  unital $u\in {CP}(A, C/\cl I)$
admits a  unital lifting
$\hat u\in {CP}(A, C )$.
 \end{cor}
\begin{proof}
  By Remarks \ref{rw} and \ref{rw1},
  a unital map $u: A \to C$  is c.p. if and only if 
  $u\in \cl F_{ump}(A, C )$.
   \end{proof}

\section{Proof of Theorem \ref{L1} }\label{pf}

We first  note that $\C=C^*(\F_\infty)$ 
satisfies the property \eqref{pty}.

\begin{lem}\label{lpty} 
For any free group $\F$ (in particular for $\F=\F_\infty$)
the $C^*$-algebra $ C^*(\F)$  
satisfies  \eqref{pty}.
\end{lem}
 \begin{proof} 
 We outline the argument already included in \cite[Th. 9.11]{P6}.
 By the main idea in \cite{Pjot} it suffices to check
 \eqref{pty'} for any $t\in {\ell_\infty(\{D_i\}) \otimes  E}$
 when $E$ is the span of the unit and a finite number 
 of the free unitary generators, say $U_1,\cdots,U_{n}$ of $C^*(\F)$
 (and actually $n=2$ suffices).
 Let $U_0=1$ for convenience.
 It is known (see \cite[p. 257-258]{P4} or \cite[p. 130-131]{P6}) that 
 for any
 $C^*$-algebra $D$ and any $(x_j)_{0\le j\le n}\in D^{n+1}$,  with respect to the inclusion $E\subset \C$, we have
 $$\|\sum x_j \otimes U_j \|_{D \otimes_{\max} E}
 =\inf \{\|\sum a_j^*a_j\|^{1/2} \|\sum b_j^*b_j\|^{1/2} \}$$
 where the infimum runs over all possible   decompositions
 $x_j=a_j^*b_j$ in $D$. Using this characterization,
 it is immediate that 
 $\|t\|_{\ell_\infty(\{D_i\}) \otimes_{\max} E} \le \sup_{i\in I} \|t_i\|_{ D_i \otimes_{\max} E}$.
  \end{proof}

   \begin{proof}[Proof of Theorem \ref{L1}]
  Assume (i).  We may assume that $A$ is unital so  that $A=\C/\cl I$ for some $\cl I$
  and  there is a unital c.p. lifting $r: A\to \C$.
  Then using   
  $$\|(Id_{D_i} \otimes r) (t_i)\|_{\max} \le \|t_i\|_{\max} $$
  for any $t_i \in {D_i} \otimes E$ (see Remark \ref{rw}),
  we easily pass from `` $\C$ satisfies (ii)"
  (which is Lemma \ref{lpty})
  to `` $A$ satisfies (ii)". This proves (i) $\Rightarrow$ (ii).\\
  Conversely,
  assume (ii). 
  By Corollary \ref{17/11},   $A$ has the $\cl F_{ucp}$-LP,
  and hence  (i) holds by Proposition \ref{14/11}.
  \\
  To prove (ii) $\Rightarrow$ (ii)$_+$, we use another fact:
  for any $x=x^*$ with $\|x\|\le 1$  in   $C$ where  $C$ is any unital $C^*$-algebra, we have
  $$x\in C_+ \Leftrightarrow \|1-x\| \le 1.$$
  Using this one easily checks  (ii) $\Rightarrow$ (ii)$_+$.
  Now assume (ii)$_+$.  Using $\|x\|=\|x^*x\|^{1/2}$
  one can reduce checking (ii) when all $t_i$'s are self-adjoint.
  Then using the same fact we obtain (ii)$_+$ $\Rightarrow$ (ii).
  \\Note: an alternate route is to use 
  $$\|x\|\le 1  \Leftrightarrow  \left\|\left(\begin{matrix} 1 \ \ x \\ x^* \ 1
  \end{matrix}\right) \right\|_{M_2(C)} \le 1 .$$
  Using this with $M_2(D_i)$ in place of $D_i$,
  one easily checks  (ii)$_+$ $\Rightarrow$ (ii).
 \end{proof}

Actually, the preceding proof can be adapted to
use only  the conclusion of Theorem \ref{t3}
to obtain the LP, therefore:

\begin{thm}\label{t22} A separable $C^*$-algebra $A$ has the LP (or equivalently the property in \eqref{pty}) if and only if 
the natural map $MB(E,C^{**}) \to MB(E,C)^{**}$ is contractive for 
all f.d. subspaces $E\subset A$ and all $C^*$-algebras $C$.
In fact,  it suffices to have this in the case $C=\C$.
\end{thm}
\begin{proof} Assume $A$ unital for simplicity. By Theorem \ref{t3}
 it suffices to show the  ``if" part. The latter is contained in
   Corollary \ref{17/11}.  To justify the last assertion,
we  reduce the LP to the case when $C$ is separable, and 
view $C$ as a quotient of $\C$.
\end{proof}

\section{LP and Ultraproducts}\label{uuty}

Let $\cl U$ be an ultrafilter on a set $I$.
Recall that the ultraproduct $(D_i)_{\cl U}$ of a family $(D_i)_{i\in I}$ of $C^*$-algebras 
is defined as the quotient 
$$ \ell_\infty(\{D_i\})/c_0^\cl U(\{D_i\})$$
where $c_0^\cl U(\{D_i\})=\{x\in \ell_\infty(\{D_i\}) \mid \lim_\cl U \|x_i\|_{D_i}
=0\}$.\\
It is usually denoted as $\prod\nl_{i\in I} D_i/\cl U$.
For short we will also denote it just  by $(D_i)_\cl U$.\\
Recall that, for each $i\in I$, we have trivially  a contractive morphism $\ell_\infty(\{D_i\}) \to D_i$ and hence also 
$\ell_\infty(\{D_i\})\otimes_{\max} A \to D_i\otimes_{\max} A $, whence a contractive morphism
\begin{equation}  \label{uty0}\Psi: \ell_\infty(\{D_i\}) \otimes_{\max} A \to \ell_\infty(\{D_i \otimes_{\max} A\})  ,\end{equation} 
which, after passing to the quotient gives us a contractive morphism
$$\ell_\infty(\{D_i\}) \otimes_{\max} A \to  (D_i \otimes_{\max} A )_\cl U  .$$
The latter morphism obviously vanishes
on $c_0^\cl U(\{D_i\}) \otimes  A$, and hence also on its closure in $\ell_\infty(\{D_i\}) \otimes_{\max} A$,
i.e. on $c_0^\cl U(\{D_i\}) \otimes_{\max} A$. Thus we obtain a natural contractive morphism
$$(\ell_\infty(\{D_i\}) \otimes_{\max} A)  / (c_0^\cl U(\{D_i\}) \otimes_{\max} A) \to
(D_i \otimes_{\max} A )_\cl U .$$
By \eqref{25/8} we have for any $C^*$-algebra $A$
$$(D_i)_{\cl U} \otimes_{\max} A=   (\ell_\infty(\{D_i\}) \otimes_{\max} A)/(c_0^\cl U(\{D_i\}) \otimes_{\max} A) , $$
whence a natural contractive morphism
\begin{equation}  \label{uty1}
\Phi: (D_i)_{\cl U} \otimes_{\max} A \to (D_i \otimes_{\max} A)_\cl U.\end{equation}

In this section we will show that  \eqref{uty1} is isometric   for any $(D_i)_{\cl U}$
if and only if $A$ has the LP.
\begin{thm}\label{ulty}
Let $A$ be a separable $C^*$-algebra.
  The following are equivalent:\\
  \item{\rm (i)} The algebra  $A$ satisfies \eqref{pty} (i.e. $A$ has the   LP).\\
  \item{\rm (ii)} For any family $(D_i)_{i\in I}$ 
 of $C^*$-algebras  and any ultrafilter on $I$ we have
 $$\forall t\in \left[\prod\nl_{i\in I} D_i/\cl U\right] \otimes A \qquad
 \|t\|_{\max}= \lim\nl_\cl U \| t_i \|_{D_i \otimes_{\max} A}.$$
 In other words we have
 a natural isometric embedding
 $$[\prod\nl_{i\in I} D_i/\cl U] \otimes_{\max} A \subset \prod\nl_{i\in I} [ {D_i \otimes_{\max} A}  ]/\cl U.$$
           \end{thm}
            \begin{proof}  Assume (i).
            If $A$ satisfies \eqref{pty},  $\Psi$ in \eqref{uty0}  is isometric.
  We claim that \eqref{uty1} must also be isometric, or equivalently injective.
Indeed, let $z\in (D_i)_{\cl U} \otimes_{\max} A$ such that $\Phi(z)=0$.
Let $z'\in \ell_\infty(\{D_i\}) \otimes_{\max} A$ lifting $z$.
Then $\Psi(z') \in c_0^\cl U(\{D_i \otimes_{\max} A\})$.
This means that for any $\vp>0$ there is a set $\alpha\in \cl U$
such that $\sup_{i\in \a} \|\Psi(z')_i\|_{D_i \otimes_{\max} A} <\vp$.
Let $P_\a: \ell_\infty(\{D_i\}) \to \ell_\infty(\{D_i\})$ be the  projection onto $\a$
defined for any $x\in \ell_\infty(\{D_i\} $  by $P_\a (x)_i=x_i$ if $i\in \a$ and
$P_\a (x)_i=0$ otherwise. Since \eqref{uty0}  is isometric we have
$$\| [P_\a \otimes Id_A](z') \|= \| \Psi ([P_\a \otimes Id_A](z') )\|=\sup\nl_{i\in \a}\|\Psi(z')_i \|_{D_i \otimes_{\max} A} \le \vp .$$
 Denoting by 
  $q: \ell_\infty(\{D_i\})\otimes_{\max} A \to (D_i)_{\cl U} \otimes_{\max} A$
  the   quotient map, since $q(z')=z$, this implies
$$\|z\|=\|q(z')\|= \|q([P_\a \otimes Id_A] (z'))\| \le \| [P_\a \otimes Id_A](z') \| \le \vp .$$
Since $\vp>0$ is arbitrary, we conclude that $z=0$, proving our claim.
This proves (i) $\Rightarrow$ (ii).\\   
Assume (ii). Consider the set $\hat I$ formed of all the finite subsets of $I$,
viewed as directed with respect to the inclusion order. 
Let $\cl U$ be a non trivial ultrafilter on $\hat I$ refining this net.
Let $d=(d_i)_{i\in I} \in \ell_\infty(\{D_i\mid i\in I\})$.
For any $J\in \hat I$ we set $D_J=  \ell_\infty(\{D_i\mid i\in J\})$.
Using natural coordinate projections we have a natural morphism $
d=(d_i)_{i\in I} \mapsto (d_J)_{J\in \hat I}\in \ell_\infty (\{D_J\mid J\in \hat I\})$. Clearly
$$\| (d_J)_\cl U\|=\lim\nl_\cl U \|d_J\|= \|d\|_{\ell_\infty(\{D_i\mid i\in I\})}.$$
Let $\psi_1: \ell_\infty(\{D_i\mid i\in I\}) \to (D_J)_\cl U$ be the corresponding
isometric $*$-homomorphism.
We claim that we have a c.p. contraction $\psi_2:(D_J)_\cl U \to 
\ell_\infty(\{D^{**}_i\mid i\in I\})$ such that the composition
$\psi_2\psi_1$ coincides with the inclusion
$\ell_\infty(\{D _i\mid i\in I\})\subset \ell_\infty(\{D^{**}_i\mid i\in I\})$.
Indeed, consider $d=(d_J)_\cl U\in (D_J)_\cl U$.
For any $J\in \hat I$ we may write $d_J=(d_J(i))_{i\in J} \in \ell_\infty(\{D _i\mid i\in J\}) $. Fix $i\in I$. We will define $\psi_2(d)$ by its $i$-th coordinate $\psi_{2i}(d)$ taking values  in $D^{**}_i$.  Note that $i\in J$ for all $J$ far enough in the net. We then set
$$\psi_{2i}(d)=\text{weak*}\lim\nl_\cl U d_J(i).$$
It is then easy to check our claim.\\
By   \eqref{30/8} the composition $\psi_2\psi_1$ is $\max$-isometric,  and  $\|\psi_2\|_{mb}\le 1$
by Remark \ref{rw},
therefore $\psi_1$ is $\max$-isometric.
In particular (taking $C=A$) $\psi_1 \otimes Id_A$ defines  an isometric embedding
$${\ell_\infty(\{D_i\mid i\in I\})} \otimes_{\max} A \subset (D_J)_\cl U \otimes_{\max} A.$$
By (ii) we have
an isometric embedding
 $(D_J)_\cl U \otimes_{\max} A \subset \prod\nl_{J\in \hat I} [ {D_J \otimes_{\max} A}  ] /\cl U.$
 It follows that for any $t=(t_i) \in {\ell_\infty(\{D_i\mid i\in I\})} \otimes_{\max} A$ we may write
 $$\|t\|_{{\ell_\infty(\{D_i\mid i\in I\})} \otimes_{\max} A}
 =\lim\nl_\cl U  \|t_J\|_{D_J \otimes_{\max} A } ,$$
 and since $\|t_J\|_{D_J \otimes_{\max} A }=\sup_{i\in J} \|t_i\|_{D_i \otimes_{\max} A }$ for any \emph{finite} $J$, 
 the inequality in \eqref{pty} follows. This proves (ii) $\Rightarrow$   (i).
 \end{proof}

 \section{The non-separable case}\label{non-sep}
  
  Let us say that a  (not necessarily separable)
  $C^*$-algebra has the LP if 
  for any separable subspace $E$ there is a separable
  $C^*$-subalgebra $A_s$ with the LP that contains $E$.\\
  For example all nuclear $C^*$-algebras
  and also $C^*(\F)$ for any free group $\F$ have the LP.
  
  To tackle the non-separable case
  the following statement 
  will be useful.
  
  \begin{pro}\label{fr18} Let $A$ be a $C^*$-algebra.
Let $E\subset A $ be a separable subspace.
 There is a separable $C^*$-subalgebra $A_s$
 satisfying $E\subset A_s\subset A$ such that,
 for any $C^*$-algebra $D$,
 the $*$-homomorphism
 $D \otimes_{\max} A_s \to D \otimes_{\max} A$
 is isometric. In other words,
 the inclusion $A_s \to A$ is $\max$-isometric.
\end{pro} 
 \begin{proof}  By Remarks \ref{r23} it suffices to prove this
 for   $D=\C$. The latter case is proved in detail in 
 \cite[Prop. 7.24]{P6} (except that the factors are flipped).
\end{proof}   

\begin{rem}\label{11/20} 
Let $A_s\subset A$ be a $C^*$-subalgebra.
Assume that the inclusion $i: A_s \to A$ is $\max$-isometric. We claim that
  for any von Neumann algebra $M$ 
the map $Id_M \otimes i: M \otimes A_s \to M \otimes A$
extends to an isometric morphism
$  M \otimes_{\rm nor} A_s \to M \otimes_{\rm nor} A$.\\
Indeed, our assumption implies that there is a
c.c.p. map (a so called ``weak expectation")
$T: A \to A_s^{**}$ such that $T_{|A_s}: A_s \to A_s^{**}$ coincides with the canonical inclusion.  See \cite[p. 88]{BO} or \cite[Th. 7.29]{P6} for a detailed proof.
By Lemma \ref{19/11} and Remark \ref{ety'} the claim follows.
\end{rem}

\begin{rem}\label{11/18} 
If $A_s \to A$ is $\max$-isometric, for any $E\subset A_s$ the spaces $MB^{E\subset A_s}(E,C)$
and $MB^{E\subset A}(E,C)$  associated to the respective embeddings
are (isometrically) identical for any $C$.
By Theorem \ref{t22}, this shows that $A$ has the LP
if and only if $\| MB (E,C^{**}) \to MB (E,C)^{**}\|=1$ for any f.d. subspace $E\subset A$.
\\Moreover, if $A$ satisfies \eqref{pty}
then $A_s$ also does. 
\end{rem}   
More precisely:
 \begin{cor}\label{20/11} A $C^*$-algebra $A$ satisfies the property in \eqref{pty}
 if and only if for any separable subspace $E\subset A$
 there is a separable $C^*$-subalgebra $A_s$ satisfying \eqref{pty}
 such that $E\subset A_s\subset A$. 
 \end{cor}
 \begin{proof} 
 Let
   $t\in \ell_\infty(\{D_i\}) \otimes E  $ with $E\subset A_s\subset A$
   as in the proposition.
  Then $t$ satisfies \eqref{pty'} with respect to $A$
   if and only if it does with respect to $A_s$.
    \end{proof}   
    
    It is now easy to extend Theorems \ref{L1} and \ref{LP2}
    to the non separable case:
    \begin{thm} A  (not necessarily separable)
  $C^*$-algebra has the LP if and only if it satisfies the 
  property \eqref{pty}.
  Moreover Theorem \ref{LP2} remains valid in the non-separable case.
   \end{thm}
    \begin{proof} The first part follows from the separable case by Corollary \ref{20/11}.
    For the second part, assume that $A$ satisfies \eqref{pty}.
    Let $t\in M \otimes E$ with $E$ f.d.
    Let $A_s$ be as in Proposition \ref{fr18}.
    Then $\|t\|_{  M \otimes_{\max}  A } =\|t\|_{  M \otimes_{\max} A_s }$.
    By Remark \ref{11/20}
    we also have $\|t\|_{  M \otimes_{\rm nor}  A } =\|t\|_{  M \otimes_{\rm nor}  A_s }$.
     By the separable case of Theorem \ref{LP2}
    we know that $\|t\|_{  M \otimes_{\rm nor}  A_s } = \|t\|_{  M \otimes_{\max}  A_s }$.
    Therefore $\|t\|_{  M \otimes_{\rm nor}  A } =\|t\|_{  M \otimes_{\max}  A }$.
    This shows that (i)' $\Rightarrow$ (ii) in
    Theorem \ref{LP2} remains valid in the non-separable case.
The other implications have already been proved there.
        \end{proof}
\begin{rem} 

We will not enumerate all the other non-separable variants 
 of the equivalent forms of the LP, except for the following
 one, that is in some sense the weakest lifting 
 requirement sufficient for the LP:
a $C^*$-algebra $A$ has the LP if and only if for any
  separable operator system $E\subset A$ and any
  $*$-homomorphism $u: A \to C/\cl I$, the restriction
  $u_{|E}$ admits a c.c.p. lifting.
  Indeed, if the latter holds
  and if $A$ is assumed unital for simplicity,
  writing $A=C^*(\F)/\cl I$ for some large enough free group $\F$,
  it is easy to deduce the desired LP for $A$ from that of 
  $C^*(\F)$. This proves the if part. The converse is clear, say,
  by Proposition \ref{14/11} applied to some separable $A_s\subset A$.
  \end{rem}

 \section{On the OLP}
            
            Let $X$ be an operator space. The associated universal unital $C^*$-algebra
            $C_u^*<X>$ is characterized by the following property:
            it contains $X$ completely isometrically,
          is generated by $X$ and the unit and any complete contraction $u: X \to C$ into a unital
       $C^*$-algebra (uniquely) extends to a unital $*$-homomorphism
       from $ C_u^*<X>$  to $C$ (see e.g. \cite[Th. 2.25]{P6}).

           Following \cite{Ozllp} we say that $X$ has the OLP if $C_u^*<X>$ has the LP.
       By the universal property of $ C_u^*<X>$, it is easy to check that
              $X$ has the OLP if and only if 
       any $u\in CB(X,C/\cl I)$ into an arbitrary quotient $C^*$-algebra
admits a lifting $\hat u\in CB(X,C )$ such that  $\|\hat u\|_{cb}=  \|  u\|_{cb}$.
See \cite[Th. 2.12]{Ozllp} for various examples of $X$ with the OLP.

            Note that,
            with respect to the inclusion $X\subset A=C_u^*<X>$,
             the norm induced  on $D \otimes X$
            by the max-norm on
        $D \otimes A$
            can be identified with the so-called $\delta$-norm (see \cite[p. 240]{P4}). We merely recall that the associated completion $D \otimes_\d X$
            can be identified as an operator space with the quotient
            $[D \otimes_h X  \otimes_h D] / \ker(Q)$ where
            $Q : D \otimes_h X  \otimes_h D \to D \otimes_\d X$
            is associated to the product map  $D \otimes  D \to D$.
            We have then isometrically when we view $X$ as sitting in $A=C_u^*<X>$:
             \begin{equation} \label{11/9/2}
            D \otimes_\d X= D \otimes_{\max} X.
            \end{equation}
            This implies that we have also isometrically for any $C^*$-algebra 
            $C$.
            \begin{equation} \label{11/9/3}
            MB(X, C)=   CB(X, C).
            \end{equation}
            Consider $E\subset X$ f.d. Let $C$ be a $C^*$-algebra.
            Recall (see Theorem \ref{ext}) that the unit ball of  the space $MB(E, C)$              
            is formed of the maps $u: E \to C$ 
            that extend to some $\dot u: A \to C^{**}$ with
            $\|\dot u\|_{dec} \le 1$.
            Let $\tilde u=\dot u_{|X}: X \to C^{**}$. Then $\|\tilde u\|_{cb}\le 1$, so 
            $u$ extends to a complete contraction $\tilde u: X \to C^{**}$.
         Conversely, if $u$ extends to a complete contraction $\tilde u: X \to C^{**}$,
         then   $\tilde u$  extends to a $*$-homomorphism $\pi: A \to C^{**}$,
         and this implies $\|u\|_{MB(E, C)} \le 1$.
            \\
            Therefore in this case
            \begin{equation} \label{11/9/1}
            \|u\|_{MB(E, C)} = \|u\|_{ext(E,C)}\end{equation} where (recall $i_C: C \to C^{**}$ denotes the canonical inclusion)
       $$\|u\|_{ext(E,C)}= \inf \{\|\tilde u\|_{cb} \mid \tilde u: X \to C^{**}, {\tilde u}_{|E}=i_C u\}.$$

Whence the following characterization 
of the OLP. The equivalence (i)  $\Leftrightarrow$ (iii) already appears in   \cite[Prop. 2.9]{Ozllp}.
\begin{thm} \label{L10}
Let $X$ be a separable operator space. The following are equivalent:
\item{\rm (i)} $X$ has the OLP.
\item{\rm (ii)}  For any  $C^*$-algebra $C$ and any f.d. subspace $E\subset X$
we have an isometric identity
$${ext(E,C^{**})} ={ext(E,C)}^{**}.$$
\item{\rm (iii)} For any  $C^*$-algebra $C$, every complete contraction $u: X \to C^{**}$ is the pointwise-weak*
limit of a net of complete contractions $u_i: X \to C$.
\item{\rm (iv)}
For any family $(D_i)_{i\in I}$ 
 of $C^*$-algebras, the natural mapping
 $${\ell_\infty(\{D_i\}) \otimes_{\d} X} \to \ell_\infty(\{D_i \otimes_{\d} X\})$$
 is isometric.
\end{thm}
\begin{proof}   
(i) $\Rightarrow$   (ii) follows from Theorem \ref{t3} and \eqref{11/9/1}. 
(ii) $\Leftrightarrow$   (iii) is easy. (i) $\Leftrightarrow$   (iii) follows
from Theorem \ref{9/10} with $\cl F$ the class of complete contractions.
(i) $\Rightarrow$   (iv) is clear by \eqref{11/9/2} since
(by Theorem \ref{L1})  the LP for  $C^*_u<X>$ implies
that it satisfies the property in
\eqref{pty}. 
\\
Assume (iv). Then by \eqref{11/9/2} and \eqref{11/9/3} we have $SB(X,C)=MB(X,C)=CB(X,C)$ isometrically for any $C$.
By Theorem \ref{lift},  (i) holds.
\end{proof}

In particular we may apply this when $X$ is a maximal operator space.
Then $\|\tilde u\|_{cb}=\|\tilde u\|$ for any $\tilde u: X \to C$ (or any $\tilde u: X \to C^{**}$).
This case is closely related to several results of Ozawa in his PhD thesis
and in \cite{Ozllp}. See also Oikhberg's \cite{[O3]}. The property (ii) in Theorem \ref{L10}
is reminiscent of Johnson and Oikhberg's extendable local reflexivity from \cite{JO}.

It gives some information on the existence of bounded liftings  
in the Banach space setting.

 \section{Illustration: Property (T) groups}\label{ill}

            To illustrate the focus our paper gives to the property in \eqref{pty} and its variants, we turn to Kazhdan's property (T), following \cite{Ozllp,Thom,[ISW]}. The proof of Theorem \ref{t11} below is merely a  reformulation of an argument
             from \cite[Th. G]{[ISW]}, but appealing to \eqref{pty} is perhaps 
             a bit quicker.

\begin{thm}[\cite{[ISW]}] \label{t11}
Let $G$ be a discrete group with property (T).
Let $(N_n)$ be an increasing sequence of normal subgroups of $G$
and let ${N_\infty}=\cup N_n$. If $C^*(G/{N_\infty})$ has the LP then
 ${N_\infty}=N_n$ for all  large enough $n$.
\end{thm}
\begin{proof}
Let $S\subset G$ be a finite unital generating subset.
Let $G_n=G/N_n$ and $G_\infty=G/{N_\infty}$.
Let $q_n: G \to G_n$  and $q_\infty: G \to G_\infty$ denote the quotient maps.
Let ${\cl M}_n=\lambda_{G_n}(G_n)''$ denote the von Neumann algebra of $G_n$, and ${\cl M}_\cl U$
its ultraproduct with respect to a non-trivial ultrafilter $\cl U$.
We will define unitary representations $\pi_n$ and $\pi_\cl U$ on $G$. We set
$\pi_n(t)= \lambda_{G_n}(q_n(t))$
for any $t\in G$. 
Let $Q_\cl U: \ell_\infty(\{ {\cl M}_n\}) \to M_\cl U$ be the quotient map.
We set
$\pi_\cl U(t)= Q_\cl U ( (\pi_n(t))   ) $.
Since the kernel of $G\ni t\mapsto \pi_\cl U(t)$ contains $N_\infty$, it defines a unitary representation on $G/N_\infty$ and hence we have
$$|S|=\|\sum\nl_S \pi_\cl U(s) \otimes \ovl{\pi_\cl U(s)}\|_{{\cl M}_\cl U \otimes_{\max}\ovl{{\cl M}_\cl U}} \le 
\|\sum\nl_S \pi_\cl U(s) \otimes \ovl{U_{G/{N_\infty}}({q_\infty(s)}) }\|_{{\cl M}_\cl U \otimes_{\max}\ovl{C^*(G/{N_\infty})}}
.$$
Let $A = C^*(G/{N_\infty})$. Using \eqref{pty} for $A $ we find
 \begin{equation} \label{11/9}
|S| \le  \lim_{n,\cl U}\|   \sum\nl_S \pi_n(s) \otimes  \ovl{U_{G/{N_\infty}}({q_\infty(s)})} \|_{  {\cl M}_n \otimes_{\max} \ovl A}.\end{equation}
Also note that the LP of $A$ implies
$${  {\cl M}_n \otimes_{\max} \ovl A}={  {\cl M}_n \otimes_{\rm nor} \ovl A}.$$
For some Hilbert space $H_n$ there is a  faithful representation
  $\sigma_n : {  {\cl M}_n \otimes_{\rm nor} \ovl A} \to B(H_n)$ 
  that is normal in the first variable.
By a well known argument \eqref{11/9} implies that the sequence 
$\{\pi_n \otimes \ovl{U_{G/{N_\infty}}\circ q_\infty} \}$, viewed as unitary representations of $G$
 in $  B(H_n)$,  admits (asymptotically) almost invariant vectors.
By property (T) for all $n$ large enough,
$\pi_n \otimes \ovl{U_{G/{N_\infty}} \circ q_\infty}$ must admit an invariant unit vector in $H_n$. 
This means   there is a unit vector
$\xi \in H_n$ such that
 $$\forall t\in G\quad \sigma_n( \pi_n(t) \otimes \ovl{U_{G/{N_\infty}} \circ q_\infty}(t) )(\xi)=\xi.$$
Recall $\pi_n=\lambda_{G_n}\circ q_n$. Let
$$\forall g\in G_n\quad f(g)= \langle \xi, \sigma_n( \lambda_{G_n}(g) \otimes 1)  (\xi)\rangle.$$
Then $f$ is a normal state on ${\cl M}_n$ such that $f(g)=1$ whenever $g\in N_\infty/N_n \subset G_n$.
 As is well known we may rewrite $f$
as $f(g)=\langle\eta , \lambda_{G_n }(g) \eta\rangle$ for some unit vector
$\eta\in \ell_2(G_n)$. Now $f(g)=1$ whenever $g\in N_\infty/N_n \subset G_n$
means that $\lambda_{G_n }(g) \eta=\eta$ whenever $g\in N_\infty/N_n \subset G_n$.
From this follows that $N_\infty/N_n$ must be finite
and hence $N_\infty$ is the union of finitely many translates of
$N_n$ by points say $t_1,\cdots,t_k$ in $N_\infty$.
Choosing $m\ge n$ so that $t_1,\cdots,t_k\in N_m$
we conclude that $N_m=N_\infty$.
\end{proof}
\begin{cor}[\cite{[ISW]}] Let $\Gamma$ be a discrete group with (T).
If $C^*(\Gamma)$ has the LP then $\Gamma$ is finitely presented.
\end{cor}
\begin{proof} 
We quickly repeat the proof in \cite{[ISW]}.
 By a result of Shalom \cite{Shalom}
 any property (T) discrete group $\Gamma$ is a quotient of a finitely presented group $G$ with property (T). 
 Enumerating the  (a priori countably many) relations that determine    $\Gamma$, we can find a sequence $N_n$ and $N_\infty$ as in Theorem \ref{t11} with $G/N_n$ finitely presented
 such that     $\Gamma=G/N_\infty$.  By Theorem \ref{t11},
 if $C^*(\Gamma)$ has the LP then
    $\Gamma=G/N_n$ for some $n$ and hence $\Gamma$ is finitely presented.
       \end{proof}

\begin{rem} Let $\cl U$ be a non-trivial ultrafilter on $\NN$.
If we do not assume the sequence $(N_n)$ nested,
let $N_\infty$ be defined by $t\in N_\infty \Leftrightarrow \lim_\cl U 1_{t\in N_n}=1$.
Then
the same argument shows
that $| q_n(\ker(q_\infty))|<\infty$ for all $n$ ``large enough" along $\cl U$.
\end{rem}

As pointed out in \cite{[ISW]}
there is a continuum of property (T) groups 
that are not finitely presented, namely those considered earlier by Ozawa in \cite{Ozpams}.
By the preceding corollary $C^*(G)$ fails the LP for all those $G$s.

\section{Some uniform estimates derived from \eqref{pty}} 

In this section our goal is to show that the property \eqref{pty}
has surprisingly strong ``local" consequences.
The results are motivated by the construction in \cite{P7}
and our hope to find
a separable $A$  with LLP but failing LP (see Remark \ref{glolo}).

   \begin{dfn}\label{rd2} Let $B,C$ be $C^*$-algebras.
            Let $ \cl E\subset B$ be a self-adjoint subspace   and let $\vp>0$. A linear map ${\psi}: \cl E \to C$ 
            will be called an $\vp$-morphism
            if  \begin{itemize}
             \item[{\rm (i)}] $\|{\psi}\|\le 1+\vp$,
             \item[{\rm (ii)}]
        for any $x,y\in \cl E$ with $xy\in \cl E$ we have
          $\|{\psi}(xy) -{\psi}(x){\psi}(y)\|\le \vp \|x\|\|y\|,$
          \item[{\rm (iii)}]
           for any $x\in \cl E$    we have        $\|{\psi}(x^*) -{\psi}(x)^*\|\le \vp \|x\|$.
           \end{itemize}
           \end{dfn}
           
 \begin{lem}\label{le1} Let $A$ be any unital $C^*$-algebra and let $D$ be another
 $C^*$-algebra. Let $t\in D \otimes A $. \begin{itemize}
 \item[{\rm (i)}] Then for any $\d>0$ 
 there is 
  $\vp>0$  and
  a f.d. operator system $\mathscr D\subset A$ such that $t\in D \otimes \mathscr D$ 
  and for any  $C^*$-algebra $C$ and any $v\in CP_\vp(\mathscr D, C)$ 
  where
  $$CP_\vp(\mathscr D, C)=\{ v: \mathscr D\to  C,  \exists v'\in CP (\mathscr D, C)\text{  with  } 
  \|v'\|\le 1 \text{  and  }  \|v-v'\|\le \vp\}$$
we have
  \begin{equation}  \label{15/8/1}\| (Id_D \otimes v)(t)\|_{\max} \le  (1+\d)    \|t\|_{\max} .\end{equation}
 \item[{\rm (ii)}]
 Assume  $t\in (D \otimes  A)^+_{\max} $.
 Then for any $\d>0$ 
 there is 
  $\vp>0$  and
  a f.d. operator system
 $\mathscr D\subset A$  
 such that $t\in D \otimes \mathscr D$
 and  for any  $C^*$-algebra $C$ and any $v\in CP_\vp(\mathscr D, C)$ we have
  \begin{equation}  \label{15/8} d( (Id_D \otimes v)(t) ,
   (D \otimes_{\max}  C)_+ ) \le \d,
 \end{equation}
 where  the distance $d$ is meant  in $D \otimes_{\max} C$.\\
 \item[{\rm (iii)}]
For any $\d>0$ there is $\vp>0$ and 
             a f.d.  subspace  $\cl E\subset A$ such that  $t\in D \otimes \cl E $ 
           and for any 
           $\vp$-morphism $\psi: \cl E \to C$ ($C$ any other $C^*$-algebra) we have
           $$   \|( Id_D\otimes \psi)(t)   \|_{D\otimes_{\max} C}\le (1+\d) \|t\|_{D\otimes_{\max} A}.$$  \end{itemize}
   \end{lem}
  \begin{proof} 
  By Remark \ref{rty0}
  we may assume $D$ unital.
  Let $t\in D \otimes A$. Let $E \subset A$ f.d. such that
  $t\in D \otimes E$.
  First note that (see Remark \ref{rty} or Proposition \ref{fr18}) there is a separable $C^*$-subalgebra $B $
  with $E\subset B\subset A$
  such that $\|t\|_{D \otimes_{\max} B}= \|t\|_{D \otimes_{\max} A}$.
  Moreover,
  assuming $\|t\|_{D \otimes_{\max} A} =1$ (for simplicity),  
   if $t\in (D \otimes A)^+_{\max} =(D \otimes A) \cap (D \otimes_{\max} A)_+ $    (which means
    that $t$ is of finite rank with $t=t^*$ and $\|1-t\|_{D \otimes_{\max} A} \le 1$) there is   a $B$
  such that in addition $t\in (D \otimes  B)^+_{\max} $.
  It follows that we may   assume $A$ separable
  without loss of generality. Then let $(E_n)$ be an increasing sequence of f.d. subspaces (or operator systems if we wish) such that $E\subset E_0$ and $A=\ovl{\cup E_n}$.
 \\
    (i) 
  We may assume by homogeneity that
  $\|t\|_{\max}=1 $. 
  We will work with the 
  sequence $(E_n,1/n)$. 
     The idea is that this sequence ``tends" to $(A,0)$.
   For any $n\in \NN$  we define $\d_n$ by
    $$1+  \d_n=\sup \| (Id_D \otimes v)(t)\|_{\max} $$   
    where the sup runs  over all  $C$'s and   $v\in CP_{1/n}(E_n, C)$.  Note that this is finite because the rank of $t$ is so.
For any $n$ there is $C_n$ and $ v_n  \in CP_{1/n}(E_n, C_n)$ such that
  \begin{equation}  \label{15/8/2}\| (Id_D \otimes v_n)(t)\|_{\max} \ge 
   1+\d_n-1/n .\end{equation} 
   Let $L=\ell_\infty(\{C_n\})$ and $\cl I=c_0(\{C_n\})$.
Consider the quotient space $\cl L=L/\cl I$
so that for any family $x=(x_n) \in \ell_\infty(\{C_n\})$
the image of $x$ under the quotient map
$Q: \ell_\infty(\{C_n\}) \to \cl L$ satisfies
$$\|Q(x)\|=\limsup\|x_n\| .$$
Note that by \eqref{25/8} we have
  \begin{equation}  \label{24/9}
  D\otimes_{\max} \cl L =
  [ D\otimes_{\max}   L]/  [ D\otimes_{\max}  \cl I]
  .\end{equation} 
  The net $(v_n) $ defines a  c.p.  map $v$
  from $A$ to  $\cl L$ with $\|v\|  \le 1$.
  It follows (see Remark \ref{rw})
that $\|( Id_D\otimes v)(t) \|_{D \otimes_{\max} \cl L} \le 1$.
By \eqref{25/8} we have
  $$\|( Id_D\otimes v)(t) \|_{D \otimes_{\max} \cl L}
=\|( Id_D\otimes v)(t) \|_{(D \otimes_{\max}   L)/(D \otimes_{\max}   \cl I) }.$$
  We have a natural morphism
   $D \otimes_{\max}   L \to D \otimes_{\max}   C_n$ for each $n$ and hence
    a natural morphism $D \otimes_{\max}   L \to \ell_\infty(\{D \otimes_{\max}   C_n\})$, taking $D \otimes_{\max}   \cl I$ to $c_0(\{D \otimes_{\max}   C_n\})$. This implies
  $$\limsup\nl_n \|  (Id_D\otimes v_n)(t) \|_{D \otimes_{\max}  C_n}\le \|( Id_D\otimes v)(t) \|_{D \otimes_{\max} \cl L} \le 1
   .$$
  Therefore by \eqref{15/8/2}
  $$\limsup\nl_n  \d_n\le 0.$$
  Thus for all $n$ large enough  
  we have $\d_n< \d $. This proves (i).
  \\
  (ii) 
  For any $n$,  let $\d_n$ denote the supremum of the left-hand side
  of \eqref{15/8} over all $C$'s and all $ v  \in CP_{1/n}(E_n, C)$.
  We introduce $v_n$ and 
   $v\in CP (A, \cl L)$ by  proceeding as in the preceding argument. Now
 $( Id_D\otimes v)(t) \in (D \otimes \cl L)^+_{\max}$,
 which, in view of \eqref{24/9},  implies
 $$\limsup\nl  d( ( Id_D\otimes v_n)(t), (D\otimes_{\max} C_n)_+)=0.$$
 We conclude as for part (i).\\
 Lastly, (iii) is proved similarly as (i).
  \end{proof}

   \begin{lem}\label{le2} Let $A$ be any unital $C^*$-algebra 
   satisfying \eqref{pty}.
   Let $E\subset A$ be a
     f.d. subspace.
  Then  for any $\d>0$ 
 there is $\vp>0$  and   a f.d. 
     operator system $\mathscr D\subset A$
     containing $E$  such that
    the following holds.\\
     (i)   For   any  $C^*$-algebras $C,D$ and any $v\in CP_\vp(\mathscr D, C)$ 
     (in particular for any $v\in CP(\mathscr D, C)$ with $\|v\|\le 1$)
we have
  \begin{equation}  \label{17/8/1}\forall t\in D \otimes E\quad 
  \| (Id_D \otimes v)(t)\|_{\max} \le  (1+\d)   \|t\|_{\max} .\end{equation}
  In other words
   \begin{equation}  \label{17/8/2}
   \|v_{|E}\|_{{MB}(E,C)}\le 1+\d.\end{equation}
 (ii) 
For any  $C^*$-algebras $C$, $D$  and any $v\in CP_\vp(\mathscr D, C)$ we have
  \begin{equation}  \label{17/8} 
  \forall t\in  (D \otimes  C)^+_{\max} \ \ 
  d( (Id_D \otimes v)(t) ,  (D \otimes_{\max}  C)_+    )\le \d \|t\|_{\max},
 \end{equation}
 where  the distance $d$ is meant  in $D \otimes_{\max} C$.\\
(iii) For any $C^*$-algebras $C$, $D$  and
           $\vp$-morphism $\psi: \mathscr D \to C$   we have
           $$   \|( Id_D\otimes \psi)(t)   \|_{D\otimes_{\max} C}\le (1+\d) \|t\|_{D\otimes_{\max} A}.$$
   \end{lem}

    \begin{proof} It is easy to see (by Remark \ref{rty} or Proposition \ref{fr18})
    that we may restrict to separable $D$'s.
    We then assemble  the set $I$ that is the disjoint union
    of the unit balls of $D\otimes _{\max} E$, when $D$ is an arbitrary
    separable $C^*$-algebra (say viewed as quotient of $\C$).
 Let $t_i\in D_i\otimes E$ be the element corresponding to $i\in I$.
 By \eqref{pty'} we know that 
 $
 \|t\|_{\ell_\infty(\{D_i\}) \otimes_{\max} A} \le 1,$
 and hence $
 \|t\|_{\ell_\infty(\{D_i\}) \otimes_{\max} E} \le 1$.
 Let $D=\ell_\infty(\{D_i\}) $ so that $t\in D \otimes E$.
 Let $\vp>0$ and  $\mathscr D\supset  E$ be associated to  $t$
 as in Lemma \ref{le1}. 
 For any $v\in CP_\vp(\mathscr D,C)$
 we have 
 $$\| (Id_D \otimes v)(t)\|_{\max} \le  (1+\d)    \|t\|_{\max} \le 1+\d.$$
 A fortiori, using the canonical morphisms $D\to D_i$
 $$\sup\nl_{i\in I} \| (Id_{D_i} \otimes v)(t_i)\|_{\max} \le 1+\d,$$
 equivalently we conclude $\|v\|_{{MB}(E,C)}\le 1+\d$.
 \\
 (ii) and (iii) are proved similarly. A priori they lead to distinct
 $\mathscr D$'s but since   replacing $\mathscr D$ by a larger f.d. one
 preserves the 3 properties, we may obtain all 3 for a common $\mathscr D$.
  \end{proof}

 \begin{rem} In the converse direction, assume that
   for any $E$ and $\d>0$  there is $\mathscr D \supset E$
 such that, for any $C$, any $v\in CP(\mathscr D, C)$ satisfies
 \eqref{17/8/2}. If $A$ (separable unital for simplicity) has the LLP then $A$ satisfies \eqref{pty}. 
 Indeed, assuming  the LLP any quotient morphism $q: \C \to A$
 will admit unital c.p. local liftings on arbitrarily large f.d. operator systems
 $\mathscr D$ in $A$. By \eqref{17/8/2}, we have local lifings
 $u^E: E \to \C$ with $\|u^E\|_{mb} \le 1+\d$ on arbitrarily large f.d. subspaces $E$ in $A$.
 From the latter it is now easy to transplant \eqref{pty}
 from $\C$ to $A$ (as in the proof of (i) $\Rightarrow$ (ii) in Theorem
 \ref{L1}). \\
 Thus the property described in part (i) of Lemma \ref{le2}
 can be interpreted as describing what is missing in the LLP
 to get the LP.
 \end{rem}
 
 Our last result is a reformulation of the equivalence of the LP
 and the property in \eqref{pty}. 
 Under this light, the LP appears as a   very strong property.
 
    \begin{thm}\label{L4} Let $A$ be a   $C^*$-algebra.
  The following are equivalent:\begin{itemize}
 \item[{\rm (i)}] The algebra  $A$  satisfies \eqref{pty} (i.e. it has the lifting property LP).
  \item[{\rm (ii)}]
 For any f.d. $E\subset A$ there is  $P=P_{E}\in \C \otimes E$ 
  with  $\|P\|_{\C \otimes_{\max} A} \le 1$
  such that
  for any unital separable $C^*$-algebra $D$ 
  and any $t\in B_{D \otimes_{\max} E}$
  there is a unital 
  $*$-homomorphism\footnote{here, actually, we could use instead
  of a $*$-homomorphism a u.c.p. map  or a $q$ with $\|q\|_{mb}\le 1$ }
  $q^D: \C \to D$ such that
  $$ [q^D\otimes Id_A](P) =t.$$
    \item[{\rm (ii)'}] For any f.d. $E\subset A$ there is a unital separable $C^*$-algebra $C$ 
 and  $P=P_{E}\in C \otimes E$ 
  with  $\|P\|_{C \otimes_{\max} A} \le 1$
  such that
such that the same as (ii) holds with $C$ in place of $\C$. 
   \item[{\rm (iii)}]
  For any f.d. $E\subset A$ and $\vp>0$
  there is  
    $P=P_{E}\in \C \otimes E$ with  $\|P\|_{\C \otimes_{\max} A} \le 1$
   such that for any unital separable $C^*$-algebra $D$
   and any $t\in B_{D \otimes_{\max} E}$
     there
   is a unital $*$-homomorphism 
  $q^D: \C \to D$ such that
  $$ \| [q^D\otimes Id_A](P) -t \|_{\max}<\vp.$$
   \item[{\rm (iii)'}] For any f.d. $E\subset A$ and $\vp>0$
  there is a unital separable $C^*$-algebra $C$
  and $P \in C \otimes E$ such that the same as (iii) holds with $C$ in place of $\C$.
  \item[{\rm (iv)}] Same as (iii)' restricted to $D=\C$.
  \end{itemize}
\end{thm}

 \begin{proof} Assume (i). 
 We view the set of all possible $D$'s as the set of quotients of $\C$.
 Let $I$ be the disjoint union of the
 sets $B_{D \otimes_{\max} E}$.
 Let $t= (t_i)_{i\in I}$ be the family of all possible $t\in B_{D \otimes_{\max} E}$. By \eqref{pty'} we have
 $$\|t\|_{\ell_\infty(\{D_i\}) \otimes_{\max} E }\le 1.$$
 Let $q: C^*(\F) \to \ell_\infty(\{D_i\})$ be a surjective $*$-homomorphism.
 Let $t' \in C^*(\F) \otimes E $  be a lifting of $t$ such that
 $\| t'\|_{ C^*(\F) \otimes_{\max} E} \le 1 $.
 There is a copy of $\F_\infty$, say $\F' \subset \F$
 such that $t'\in C^*(\F') \otimes E$. 
 Then $C^*(\F') \subset C^*(\F)$. Let $C=C^*(\F')$.
 We can take for $q^D$ the restriction of $q_i$ to $C^*(\F') \simeq \C$,
 and we set  $P= t'$ viewed as sitting in $\C \otimes E$.
This gives us (i) $\Rightarrow$ (ii) and (ii) $\Rightarrow$ (ii)' is trivial.
\\
 (ii) $\Rightarrow$ (iii) and  (iii) $\Rightarrow$  (iii)'  $\Rightarrow$(iv) are trivial.
 \\
Assume (iii)'.  Let $Q: \C \to C$  be a quotient unital morphism.
Let $P\in  C \otimes E$ be as in (iii)'. Let $P'\in \C \otimes E$
with $\|P'\|_{\max} \le 1$ be such that $[Q\otimes Id_E](P') =P$
(this exists by the assertion following \eqref{7/9/1}).
Then $P'$ has the property required  in (iii).
Thus (iii)' $\Rightarrow$ (iii). Similarly  (ii)' $\Rightarrow$ (ii). 
\\
The implication
(iv) $\Rightarrow$ (iii)' is easy to check similarly using the fact that any
$D$ is a quotient of $\C$. We skip the details.
It remains to show (iii) $\Rightarrow$ (i).

Assume (iii).  We will show \eqref{pty'}.
We need a preliminary elementary observation.
Let $(e_j)$ be a normalized algebraic basis of $E$.
Then it is easy to see that there is a constant $c$ (depending on $E$ and $(e_j)$) such that for any $D$ and any $d_j,d_j'\in D$ we have
 \begin{equation}\label{20/9}
 \sup\nl_j\|d_j-d_j'\| \le c \|\sum (d_j-d_j') \otimes e_j \|_{D \otimes_{\max} E}.\end{equation}
Now let $t_i\in B_{D_i \otimes_{\max} E}$. Let $P\in B_{\C \otimes_{\max}  E}$
be as in (iii).
For some quotient morphism $q^i: \C \to D_i$ we have
$$\|t_i-[q^i\otimes Id_A](P)\|_{ \C  \otimes_{\max} E} \le \vp .$$
Let $P_i=P$ and $\C_i=\C$ for all $i\in I$.
Clearly (since $P\mapsto (P_i)$ is a $*$-homomorphism on $\C \otimes A$)
$$\|(P_i)\|_{\ell_\infty(\{\C_i\})  \otimes_{\max} E} \le \|P \|_{\C  \otimes_{\max} E}
\le 1.$$
Note $(q^i): \ell_\infty(\{\C_i\}) \to \ell_\infty(\{D_i\})$
is a $*$-homomorphism.
Let $(t'_i)= [q^i\otimes Id_A](P_i)$.
We have $$\|(t'_i)\|_{\ell_\infty(\{D_i\})  \otimes_{\max} E} \le 1
\text{  and }\sup\nl_i\|  t_i-t'_i\|_{ \C  \otimes_{\max} E}\le \vp ,$$
By   the triangle inequality and \eqref{20/9}  we have
$\| (t_i)-(t'_i)\|_{\ell_\infty(\{D_i\})  \otimes_{\max} E} \le c\dim(E) \vp$. 
Therefore
$\| (t_i) \|_{\ell_\infty(\{D_i\})  \otimes_{\max} E} \le 1+c\dim(E) \vp$.
Since $\vp>0$ is arbitrary for any fixed $E$, this means that  \eqref{pty'} holds whence
(i) (and the LP   by Theorem \ref{L1}). 
 \end{proof}
 In passing we note :
  \begin{cor} For the property in \eqref{pty}
  it suffices to check the case when $I=\N$   (and $D_i=\C$ for all $i\in I$).
  \end{cor}
   \begin{proof} We already observed that $D_i=\C$ for all $i\in I$ suffices (see Proposition \ref{rty1}).
   Assume that \eqref{pty'} holds for $I=\N$ and $D_i=\C$.
   In the proof of (i) $\Rightarrow$ (ii) in Theorem \ref{L4} we may use for the set $I$ the 
   disjoint union of dense sequences
   in $B_{D_k \otimes_{\max} E}$ with $k\in \N$ and $D_k=\C$. Then $I$ is countable
   and with the latter $I$ the proof of (i) $\Rightarrow$ (ii) in Theorem \ref{L4}
   gives us (iii) in Theorem \ref{L4}.
 By (iii) $\Rightarrow$ (i) in Theorem \ref{L4}, the corollary follows.
   \end{proof}
 \begin{rem}
 Let $\cl P$ denote the (non-closed) $*$-algebra generated by the free unitary generators of $\C$.
 Let $P\in \C \otimes E$ be as in (ii) in Theorem \ref{L4}.
 Note for any $\vp >0$ there is $Q\in \cl P \otimes E$
 with $\| Q-P\|_{\wedge} <\vp$ (where $\|\ \|_{\wedge} $ denotes the projective norm), and hence
   also
 $\|q^D(Q)-t\|_\wedge <\vp$.
 After a suitable normalization we may
  also assume $\|Q\| _{\max} < 1$.
 Then $Q$ has a factorization
 in products in $\cup M_n( \C) $ and $\cup M_n( A) $
 and a further Blecher-Paulsen type factorization 
 in the style described in \cite[\S 26]{P4}. Thus we find a \emph{common}
 pattern in the approximate factorization
 of \emph{all} the elements of  $B_{D \otimes_{\max} E}$.
 
 The simplest illustration of this phenomenon is the case when
 $A=\C$ and $E$ is the span of the unitary generators $(U_i)$.
 In that case we have   $t=\sum d_i \otimes U_i\in B_{D \otimes_{\max} E}$ if and only if
 for any $\vp>0$ there is a factorization of the form
 $d_i=a_i b^*_i$ with $a_i,b_i\in D$ such that $\|\sum a_i a_i^*\|^{1/2}=\|\sum b_ib_i^*\|^{1/2}<1+\vp$ (see \cite[p. 130-131]{P6}).
 Equivalently, this holds if and only if
 for any $\vp>0$ there  are $a_i,b_i\in D$
 such that $\|\sum a_i a_i^*\|^{1/2}=\|\sum b_ib_i^*\|^{1/2} =1$ 
and 
 $\sum\nl_i\|d_i-a_i b^*_i\|<\vp$. 
 Let us assume that the latter holds.
  Consider first the column operator space $C_n$, then its universal
  $C^*$-algebra $C^*_u< C_n>$, and lastly
   the free product $C=C^*_u< C_n>  \ast \ C^*_u< C_n>$. Let $e_{i 1}^{(1)}\in C$ and
  $e_{i 1}^{(2)}\in C$ denote the  natural basis of the column space $C_n$ of  each free factor.
  Let $P_E=\sum e_{i 1}^{(1)}e_{i 1}^{(2)^*} \otimes  U_i\in C\otimes E$.
  By our assumption on $(d_i)$, there is a $*$-homomorphism
  $q^D: C \otimes E$ such that $q^D(e_{i 1}^{(1)})=a_i$
  and $q^D(e_{i 1}^{(2)})=b_i$ and hence
  $\|q^D(P_E)-\sum d_i \otimes U_i\|\le \vp$.
  This illustrates the property (iii)' from Theorem \ref{L4}.
  Note in passing that since the column space
  $C_n$ has the OLP, the free product 
  $C=C^*_u< C_n>  \ast \ C^*_u< C_n>$ has the LP.
  The latter algebra  can be substituted with $\C$ in many questions involving tensor products.
 \\
 Using the preceding remark, one can give an independent proof 
 of Lemma \ref{le2}.
  \end{rem}

    \medskip
    
    \n\textbf{Acknowledgement.} I am grateful to Jean Roydor
   and Mikael de la Salle  for stimulating and useful comments.

\end{document}